\documentclass[pdflatex,sn-mathphys-num]{sn-jnl}

\usepackage{graphicx}%
\usepackage{multirow}%
\usepackage{amsmath,amssymb,amsfonts}%
\usepackage{amsthm}%
\usepackage{mathrsfs}%
\usepackage[title]{appendix}%
\usepackage{xcolor}%
\usepackage{textcomp}%
\usepackage{manyfoot}%
\usepackage{booktabs}%
\usepackage{algorithm}%
\usepackage{algorithmicx}%
\usepackage{algpseudocode}%
\usepackage{listings}%
\usepackage{empheq}
\usepackage{enumitem}
\usepackage{pdflscape}   
\usepackage{xltabular}   
\usepackage{rotating}
\usepackage{tabularx}
\usepackage{array}
\usepackage{adjustbox}

\theoremstyle{thmstyleone}%
\newtheorem{theorem}{Theorem}[section]
\newtheorem{proposition}[theorem]{Proposition}%
\newtheorem{assumption}[theorem]{Assumption}
\newtheorem{corollary}[theorem]{Corollary}
\newtheorem{lemma}[theorem]{Lemma}

\theoremstyle{thmstyletwo}%
\newtheorem{example}[theorem]{Example}%
\newtheorem{remark}[theorem]{Remark}%

\theoremstyle{thmstylethree}%
\newtheorem{definition}[theorem]{Definition}%

\begin{document}

\title[Primal--Dual Error Bounds and KKT Metric Subregularity]{Primal-Dual Error Bounds and KKT Metric Subregularity in Convex Composite Optimization}


\author[1]{\fnm{Jiani} \sur{Li}}\email{lijiani\_qaz@163.com}

\author*[1,2]{\fnm{Qingna} \sur{Li}}\email{qnl@bit.edu.cn}


\affil*[1]{\orgdiv{School of Mathematics and Statistics}, \orgname{Beijing Institute of Technology}, \orgaddress{\city{Beijing}, \postcode{100081}, \country{China}}}

\affil[2]{\orgdiv{Beijing Key Laboratory on MCAACI/Key Laboratory of Mathematical Theory and Computation in Information Security}, \orgname{Beijing Institute of Technology}, \orgaddress{\city{Beijing}, \postcode{100081}, \country{China}}}



\abstract{
	We study primal and dual error bounds for convex composite optimization	problems and their Fenchel duals. For the primal problem, we use a	proximal-gradient residual, while for the generally nonsmooth dual	problem we use a subdifferential residual. We introduce a reduced	primal--dual KKT mapping whose zero set coincides with the Cartesian	product of the primal and dual solution sets under the standing	solvability and strong-duality assumptions.
	
	We establish implications in both directions between metric	subregularity of the reduced KKT mapping and the primal and dual	  local error bounds. If the composite matrix has full row	rank and the smooth gradient is locally Lipschitz continuous, KKT	metric subregularity is equivalent to the primal  local error bound. If the smooth term is strongly convex, it is equivalent to the	dual local error bound. Under full row rank and strong convexity,	KKT metric subregularity is equivalent to the simultaneous validity	of both error bounds.
	
	We also establish relations between the Luo--Tseng and local error 	bounds. The framework is then applied to quadratic--polyhedral models, 	strongly convex smooth models with general convex regularizers, and 	regularized least-squares models beyond polyhedrality and strong 	convexity, providing a unified treatment of primal, dual, and KKT 	error-bound properties.
}




\maketitle

\section{Introduction}
\label{sec:introduction}
In this paper, we study error bounds for the convex composite problem
\begin{equation}\label{eq:intro-primal}
\min_{v\in\mathcal{X}}\;
F(v):=f_1(v)+f_2(Av+b),
\end{equation}
where \(A:\mathcal{X}\to Y\), \(b\in Y\), \(f_1:\mathcal{X}\to\mathbb R\) is finite-valued, convex, and Fr\'echet differentiable, and \(f_2:Y\to(-\infty,+\infty]\) is proper, closed, and convex. Its Fenchel dual, written in minimization form, is
\begin{equation}\label{eq:intro-dual}
	\min_{x\in Y}\;
	d(x):=
	f_2^*(x)+f_1^*(-A^\top x)-\langle b,x\rangle .
\end{equation}

The composite structure above covers a broad range of models arising in statistics, machine learning, and so on. Representative examples include the Lasso type problem  \cite{Osborne2000}, the support vector machine  \cite{CortesVapnik1995}, and convex clustering  \cite{Pelckmans2005}.  

Error bounds estimate the distance from a candidate point to a solution set by means of a residual measuring violation of an optimality condition. They play an important role in the convergence analysis of optimization algorithms and, in particular, often provide linear convergence guarantees without requiring global strong convexity. A classical starting point is Hoffman's error bound for systems of linear inequalities~\cite{Hoffman1952}; related developments for variational inequalities and polyhedral multifunctions were given by Pang~\cite{Pang1987} and Robinson~\cite{Robinson2009}. Since then, error bounds have been extensively studied and increasingly used in the convergence analysis of optimization algorithms. Below, we review the error-bound results most closely related to our work.

\subsection{Related work}
In the early 1990s, Luo and Tseng~\cite{LuoTseng1992} studied convex essentially smooth minimization over a polyhedral set and established an error bound estimating the distance from a feasible point to the optimal solution set, which was used to prove linear convergence of descent methods. They subsequently developed a more general error-bound analysis for feasible descent methods~\cite{LuoTseng1993a}. Later, Luo and Tseng~\cite{LuoTseng1993b} analyzed dual ascent methods for linearly constrained convex minimization and derived linear convergence from an estimate of the distance from a feasible dual point to the dual solution set by a residual. Error-bound analysis was later extended to nonsmooth and structured optimization models. Tseng and Yun~\cite{TsengYun2009} studied minimization of a smooth function plus a separable convex function and established linear convergence of a block coordinate gradient method under a local Lipschitzian error-bound condition. Tseng~\cite{Tseng2010} further studied error bounds and gradient methods for structured convex optimization. Error bounds and convergence rates were also investigated for specific regularization structures, including \(\ell_{1,p}\)-regularized problems~\cite{ZhouZhangSo2015} and trace-norm regularization~\cite{HouZhouSoLuo2013}.

Related work has also developed more general characterizations and frameworks for error bounds. Meng and Yang~\cite{MengYang2012} established two classes of equivalent conditions for local error bounds in finite-dimensional spaces, using subderivatives, subdifferentials, strong slopes, and variational properties of the reference set. For structured convex optimization, Zhou and So~\cite{zhouso2017} developed a unified framework for establishing error bounds for objectives consisting of a smooth convex function and a general proper closed convex function, and showed that a number of existing error-bound results can be recovered within this framework. Connections between error bounds and growth properties have also been studied through variational analysis. Arag\'on Artacho and Geoffroy~\cite{AragonGeoffroy2008} characterized metric regularity and subregularity of convex subdifferentials in terms of growth conditions, while Drusvyatskiy and Lewis~\cite{DrusvyatskiyLewis2018} established an equivalence between a proximal error bound and quadratic growth and used it to explain linear convergence of proximal methods. Necoara, Nesterov, and Glineur~\cite{NecoaraNesterovGlineur2019} established linear convergence of several first-order methods under conditions weaker than global strong convexity. More recently, Atenas et al.~\cite{AtenasSagastizabalSilvaSolodov2023} derived linear rates for classes of descent methods under a subdifferential-based error bound, while Traor\'e, Salzo, and Villa~\cite{TraoreSalzoVilla2023} obtained linear convergence of an asynchronous block-coordinate forward-backward method under an error-bound condition of Tseng type. Error bounds have also been studied for nonsmooth variational inequality models; see, for example, Li, Meng, and Yang~\cite{LiMengYang2025}.

On the other hand, error bounds have been closely connected with regularity properties of set-valued mappings and generalized equations. Robinson~\cite{Robinson1980} introduced the strong regularity framework for generalized equations. For convex generalized equations, Zheng and Ng~\cite{ZhengNg2007} studied metric subregularity and its relation to constraint qualifications, and Zheng and Ouyang~\cite{ZhengOuyang2011} extended related metric-subregularity results to composite-convex generalized equations. First order and second order characterizations of metric subregularity and calmness were developed by Gfrerer~\cite{Gfrerer2011}, while Kruger~\cite{Kruger2015} treated necessary and sufficient criteria for metric subregularity within an error-bound framework. Strong metric subregularity of set-valued mappings was studied systematically by Cibulka, Dontchev, and Kruger~\cite{CibulkaDontchevKruger2018}. For optimization systems, Ding, Sun, and Zhang~\cite{DingSunZhang2017} investigated robust isolated calmness for a class of conic programming problems, Bai, Ye, and Zhang~\cite{BaiYeZhang2019} developed directional quasi-/pseudo-normality conditions sufficient for metric subregularity. Closer to the setting of the present paper, Burke and Engle~\cite{BurkeEngle2020} embedded the optimality conditions of piecewise linear-quadratic convex-composite problems into a generalized equation and established conditions for strong metric subregularity and strong metric regularity of the corresponding KKT mappings. Mohammadi, Mordukhovich, and Sarabi~\cite{MohammadiMordukhovichSarabi2022} developed variational analysis of composite models with an emphasis on replacing metric-regularity assumptions by weaker metric-subregularity conditions, with applications including KKT systems. In convex composite conic programming, Cui, Sun, and Toh~\cite{CuiSunToh2019} established superlinear convergence of KKT residuals for the augmented Lagrangian method under a quadratic growth condition on the dual problem.

\subsection{Main contributions and outline}

We first introduce a reduced KKT mapping for the primal--dual pair \eqref{eq:intro-primal}--\eqref{eq:intro-dual}. Under the standing solvability and strong-duality assumptions, its zero set is exactly the Cartesian product of the primal and dual solution sets. This provides a common stability framework for the primal proximal-gradient residual, the dual subdifferential residual, and the KKT residual. We establish implications from KKT metric subregularity to the primal and dual   local error bounds under suitable exact-calculus and localization conditions. Conversely, the primal and dual  local error bounds together imply KKT metric subregularity under suitable reverse residual comparisons. We further show that a one-sided   local error bound is sufficient when supplemented by an appropriate distance-control condition.

These abstract implications lead to simple equivalence results under structural assumptions. If \(A\) has full row rank and \(\nabla f_1\) is locally Lipschitz continuous around a primal solution, KKT metric subregularity is equivalent to the primal  local error bound. If \(f_1\) is strongly convex, KKT metric subregularity is equivalent to the dual   local error bound.
 Moreover, if \(A\) has full row rank and \(f_1\) is strongly convex, KKT metric subregularity is equivalent to the simultaneous validity of the primal and dual   local error bounds. We also establish the relations between the  local and Luo--Tseng error bounds: the Luo--Tseng bounds yield the corresponding local bounds, while level-boundedness allows local bounds valid at every solution to be extended to the Luo--Tseng bounds.

Finally, we apply the framework to quadratic--polyhedral models, strongly convex smooth models with general convex regularizers, and regularized least-squares models whose smooth term may fail to be strongly convex and whose regularizer may be nonpolyhedral. The quadratic--polyhedral results cover Lasso-type models, SLOPE, several SVC and SVR formulations, and polyhedrally constrained convex quadratic programs. The strongly convex regime includes weighted convex clustering. Beyond these two structural regimes, we use metric-subregularity results for \(\ell_{1,q}\) and sparse-group regularizers to derive primal and dual error bounds for possibly rank-deficient regularized least-squares models, including Group Lasso.

The remainder of the paper is organized as follows.
Section~\ref{sec:Preliminaries} develops the preliminary primal--dual framework and introduces the error-bound properties used throughout the paper.
Section~\ref{sec:primal-dual-eb-relation} presents the main theoretical results relating KKT metric subregularity to the primal and dual   local error bounds and derives equivalence results under structural conditions.
Section~\ref{sec:applications} applies these results to
quadratic--polyhedral models, strongly convex composite problems, and regularized least-squares models whose reduced KKT mapping is known to be metrically subregular.
Section~\ref{sec:conclusion} concludes the paper.

\subsection{Notation}
The following notation is used throughout the paper. All vector spaces are finite-dimensional Euclidean spaces. For vectors \(z\) and \(w\) in the same Euclidean space, \(\langle z,w\rangle\) and \(\|z\|\) denote the standard Euclidean inner product and norm, respectively. On \(\mathbb R^n\times\mathbb R^m\), we use the product norm
$\|(v,x)\|:=\bigl(\|v\|^2+\|x\|^2\bigr)^{1/2}.$ For a matrix \(A\), \(A^\top\), \(A^\dagger\), and \(\|A\|\) denote its transpose, Moore-Penrose pseudoinverse, and induced Euclidean operator norm, respectively. For a symmetric matrix \(Q\), we write \(Q\succeq0\) if \(Q\) is positive semidefinite, that is, \(\langle z,Qz\rangle\geq0 \ \text{for all }z.\)

For a point \(z\) in a Euclidean space \(\mathcal E\) and \(r>0\), the closed ball centered at \(z\) with radius \(r\) is $\mathbb B_{\mathcal E}(z,r):=\{w\in\mathcal E:\|w-z\|\leq r\}.$ We omit the subscript \(\mathcal E\) when the ambient space is clear. For a set \(C\subseteq\mathcal E\), we write $\operatorname{dist}(z,C):=\inf_{w\in C}\|z-w\|,$ with the convention $ \operatorname{dist}(z,\emptyset):=+\infty.$ If \(C\) is nonempty, closed, and convex, \(\operatorname{proj}_C(z)\) denotes the Euclidean projection of \(z\) onto \(C\).

For an extended-real-valued function \(h\), \(\operatorname{dom}h\), \(\arg\min h\), \(h^*\), \(\partial h\), and \(\operatorname{prox}_h\) denote its effective domain, solution set, Fenchel conjugate, convex subdifferential, and proximal mapping, respectively. A proper convex function \(h\) is said to be \(\sigma\)-strongly convex,
for some \(\sigma>0\), if \(h(tx+(1-t)y) \leq t h(x)+(1-t)h(y) -\frac{\sigma}{2}t(1-t)\|x-y\|^2 \) for all \(x,y\in\operatorname{dom}h\) and \(t\in[0,1]\).  For a set-valued mapping \(\Phi:\mathcal E\rightrightarrows\mathcal T\), we use $\Phi^{-1}(y):=\{z\in\mathcal E:y\in\Phi(z)\}.$

\section{Preliminaries and problem setting}\label{sec:Preliminaries}
This section introduces the primal--dual framework and the error-bound properties used throughout the paper. We first formulate the Fenchel dual and the reduced KKT mapping, then define the primal and dual residuals together with the corresponding Luo--Tseng error bound and local error bound. Finally, we clarify how the local error bounds used in this paper are related to the Luo--Tseng  error bounds.
\subsection{Primal and dual problems and reduced KKT mapping}\label{subsec:primal-dual-problems}

We now formalize the primal-dual setting used throughout the paper. Consider the convex optimization problem
\begin{equation}\label{eq:P}
    \min_{v\in\mathbb R^n}\;
    F(v):=f_1(v)+f_2(Av+b),
    \tag{P}
\end{equation}
where \(A\in\mathbb R^{m\times n}\), \(b\in\mathbb R^m\), $f_1:\mathbb R^n\to\mathbb R$ is finite-valued, convex, and Fréchet differentiable, and $f_2:\mathbb R^m\to(-\infty,+\infty] $ is proper, closed, and convex.

Using the conjugate representation of \(f_2\), the Fenchel dual of \eqref{eq:P} is
\[
    \sup_{x\in\mathbb R^m}
    \left\{
        \langle b,x\rangle
        -
        f_2^*(x)
        -
        f_1^*(-A^\top x)
    \right\}.
\]
We write the dual equivalently in minimization form as
\begin{equation}\label{eq:Dmin}
    \min_{x\in\mathbb R^m}\;
    d(x):=f_2^*(x)
    + f_1^*(-A^\top x) -
    \langle b,x\rangle.
    \tag{D}
\end{equation}

Since \(f_1\) is finite-valued and convex on \(\mathbb R^n\), it is continuous and proper. Hence, by \cite[Theorem~12.2]{Rockafellar1970}, \(f_1^*\) is proper, closed, and convex. Similarly, \(f_2^*\) is proper, closed, and convex. Consequently, the dual objective \(d\) is closed and convex, although its properness is not automatic. Properness is therefore included in the standing assumptions below.

Define the primal and dual solution sets by
$\mathcal V:=\arg\min F, \  \Lambda:=\arg\min d,$ and denote the corresponding optimal values by $F^\star :=\inf_{v\in\mathbb R^n}F(v), \ d^\star:=\inf_{x\in\mathbb R^m}d(x).$ With the dual written in the minimization form \eqref{eq:Dmin}, strong duality is expressed as \(F^\star+d^\star=0\).

The standing solvability and duality assumptions used throughout the paper are summarized below.

\begin{assumption}[Basic primal-dual setting]\label{ass:basic-primal-dual-setting}
The following conditions hold.

\begin{enumerate}[label=\textnormal{(\alph*)},leftmargin=*]

\item The primal objective \(F\) is proper and
\(\mathcal V\neq\emptyset\).


\item The dual objective \(d\) is proper and \(\Lambda\neq\emptyset\).

\item Strong duality holds: \(F^\star+d^\star=0\).

\end{enumerate}
\end{assumption}

Unless otherwise stated, Assumption~\ref{ass:basic-primal-dual-setting} is maintained throughout the paper. Under this assumption, \(F\) and \(d\) are proper, closed, and convex. Indeed, \(f_2\circ(A\cdot+b)\) is proper, closed, and convex, while \(f_1\) is finite-valued and continuous. Consequently, the solution sets \(\mathcal V\) and \(\Lambda\) are nonempty, closed, and convex.

By the Fenchel--Young equality and the conjugacy relation for convex subdifferentials, the primal and dual optimality conditions can be expressed as
\[
    -A^\top x\in\partial f_1(v),
    \qquad
    Av+b\in\partial f_2^*(x).
\]
Since \(f_1\) is differentiable, this leads to the reduced primal-dual KKT system
\begin{equation}\label{eq:reduced-kkt-current}
    0=\nabla f_1(v)+A^\top x,
    \qquad
    Av+b\in\partial f_2^*(x).
\end{equation}

We refer to \eqref{eq:reduced-kkt-current} as the reduced primal-dual KKT system and define its solution set by

\[
    \mathcal S :=
    \left\{
        (v,x)\in\mathbb R^n\times\mathbb R^m \mid 
        \nabla f_1(v)+A^\top x=0,\;
        Av+b\in\partial f_2^*(x)
    \right\}.
\]
The next result shows that this reduced system exactly characterizes primal and dual optimality under the standing assumptions.

\begin{proposition}\label{prop:kkt-solution-product}

Under Assumption \ref{ass:basic-primal-dual-setting}, the reduced KKT solution set satisfies
\[
    \mathcal S=\mathcal V\times\Lambda.
\]
\end{proposition}
\begin{proof}
	See Appendix~\ref{app:kkt-solution-product}.
\end{proof}

This Cartesian-product characterization makes the reduced KKT system a natural common stability object for the primal and dual problems. We therefore define the set-valued mapping
\begin{equation}\label{eq:reduced-kkt-mapping}
    \mathcal M(v,x)
    :=
    \begin{pmatrix}
        \nabla f_1(v)+A^\top x \\
        Av+b-\partial f_2^*(x)
    \end{pmatrix}.
\end{equation}
Here \(\mathcal M:\mathbb R^n\times\mathbb R^m
\rightrightarrows \mathbb R^n\times\mathbb R^m\) is a set-valued mapping, and
\[
    0\in\mathcal M(v,x)
\]
is exactly the reduced KKT system \eqref{eq:reduced-kkt-current}. Hence
\begin{equation}\label{eq:S-M-inverse}
    \mathcal S=\mathcal M^{-1}(0).
\end{equation}
The corresponding KKT residual is
\begin{equation}\label{eq:kkt-residual-current}
    r_{\rm KKT}(v,x)
    :=
    \operatorname{dist}(0,\mathcal M(v,x)).
\end{equation}
Equivalently,
\[
    r_{\rm KKT}(v,x)
    =
    \left(
    \|\nabla f_1(v)+A^\top x\|^2
    +
    \operatorname{dist}^2
    \bigl(Av+b,\partial f_2^*(x)\bigr)
    \right)^{1/2}.
\]
We measure the local stability of the reduced KKT system through metric subregularity of \(\mathcal M\); see, e.g., \cite{DontchevRockafellar2014,Kruger2015}.

\begin{definition}[Metric subregularity of the reduced KKT mapping]
\label{def:kkt-msr-pointwise}
Let \((\bar v,\bar x)\in\mathcal S\). We say that the reduced KKT mapping \(\mathcal M\) is metrically subregular at \((\bar v,\bar x)\) for \(0\) if there exist constants \(\kappa_{\rm KKT}>0\) and \(\epsilon_{\rm KKT}>0\) such that
\begin{equation}\label{eq:kkt-msr-pointwise}
    \operatorname{dist}((v,x),\mathcal S)
    \leq
    \kappa_{\rm KKT}\operatorname{dist}(0,\mathcal M(v,x)),
    \qquad
    \forall (v,x)\in
    \mathbb B((\bar v,\bar x),\epsilon_{\rm KKT}).
\end{equation}
\end{definition}

In particular, by Proposition~\ref{prop:kkt-solution-product}, the Euclidean product norm gives
\begin{equation}\label{eq:kkt-product-distance}
\operatorname{dist}^2\bigl((v,x),\mathcal S\bigr)
=
\operatorname{dist}^2(v,\mathcal V)
+
\operatorname{dist}^2(x,\Lambda),
\qquad
(v,x)\in\mathbb R^n\times\mathbb R^m.
\end{equation}
Together with \(\mathcal M^{-1}(0)=\mathcal S\), this identity shows that metric subregularity of the reduced KKT mapping controls the primal and dual distances simultaneously through the KKT residual. This provides the basis for relating KKT stability to the primal and dual error bounds studied in the following sections.


\subsection{Residual mappings and error bounds}\label{subsec:residual-error-bounds}

Having introduced the reduced KKT residual in the previous section, we now define the residuals associated separately with the primal and dual problems and formulate the corresponding error-bound properties. Let $G(v):=f_2(Av+b).$  Recall that \(G=f_2\circ(A\cdot+b)\) is proper, closed, and convex. Then problem~\eqref{eq:P} can be written as
\[
    \min_{v\in\mathbb R^n}\; f_1(v)+G(v).
\]
Define the primal residual mapping by
\begin{equation}\label{eq:primal-residual-map}
     r_P(v)
    :=
    \operatorname{prox}_{G}
    \bigl(v-\nabla f_1(v)\bigr)-v,
\end{equation}
where
\[
    \operatorname{prox}_{G}(z)
    :=
    \arg\min_{u\in\mathbb R^n}
    \left\{
        \frac12\|u-z\|^2+G(u)
    \right\}.
\]
Since \(G\) is proper, closed, and convex, \(\operatorname{prox}_G\) is nonexpansive \cite[Proposition~12.28]{BauschkeCombettes2017}. Moreover, since \(f_1\) is finite-valued, convex, and differentiable on \(\mathbb R^n\), its gradient \(\nabla f_1\) is continuous; see \cite[Corollary~25.5.1]{Rockafellar1970}. Consequently, \(v\mapsto v-\nabla f_1(v)\) is continuous, and the nonexpansiveness of \(\operatorname{prox}_G\) implies that \( r_P\) is continuous on \(\mathbb R^n\).

The dual objective \(d\) is generally nonsmooth. In particular, \(f_2^*\) may be nonsmooth, and without additional assumptions on \(f_1\), the conjugate \(f_1^*\) need not be differentiable either. We therefore use the subdifferential residual
\begin{equation}\label{eq:dual-subgradient-residual}
    r_D(x) := \operatorname{dist}(0,\partial d(x)).
\end{equation}
By the optimality condition of the proximal mapping and the first-order optimality condition for convex minimization, \( r_P(v)=0\) if and only if \(v\in\mathcal V\), while \(r_D(x)=0\) if and only if \(x\in\Lambda\).

Following the general local error-bound formulation of Meng and Yang~\cite{MengYang2012}, we use the term   local error bound for the residual-specific local estimates introduced below.

\begin{definition}[Local error bound {\cite{MengYang2012}}]
\label{def:general-local-error-bound}
Let \(C\subseteq\mathbb R^p\) be a nonempty closed set, and let \(\phi:\mathbb R^p\to\mathbb R_+\cup\{+\infty\}\) satisfy \(\phi(z)=0\) if and only if \(z\in C\). Given \(\bar z\in C\), the function \(\phi\) is said to be a local error bound of \(C\) at \(\bar z\) if 
there exist constants \(\kappa>0\) and \(\rho>0\) such that
\[
    \operatorname{dist}(z,C)
    \le
    \kappa\,\phi(z),
    \qquad
    \forall z\in\mathbb B(\bar z,\rho).
\]
\end{definition}

We now specialize the above local error-bound formulation to the primal and dual problems.

\begin{definition}[Local error bounds for the primal and dual problems]
	\label{def:local-primal-prox-gradient-eb}
	\label{def:local-dual-subdifferential-eb}
	\leavevmode\\ \textnormal{(i)}
	Let \(\bar v\in\mathcal V\). We say that the   local
	error bound holds for the primal problem \eqref{eq:P} at
	\(\bar v\) if there exist constants \(\rho_P>0\) and
	\(\kappa_P>0\) such that
	\[
	\operatorname{dist}(v,\mathcal V)
	\leq
	\kappa_P\| r_P(v)\|,
	\qquad
	\forall v\in\mathbb B(\bar v,\rho_P).
	\]
	
	\noindent \textnormal{(ii)}
	Let \(\bar x\in\Lambda\). We say that the   local
	error bound holds for the dual problem \eqref{eq:Dmin} at
	\(\bar x\) if there exist constants \(\rho_D>0\) and
	\(\kappa_D>0\) such that
	\[
	\operatorname{dist}(x,\Lambda)
	\leq
	\kappa_D r_D(x),
	\qquad
	\forall x\in\mathbb B(\bar x,\rho_D).
	\]
\end{definition}

Following the error-bound framework of Luo and Tseng~\cite{LuoTseng1993a} and the proximal-residual formulation used by Zhou and So~\cite{zhouso2017}, we recall the following Luo--Tseng error bound for the primal problem. The terminology ``Luo--Tseng error bound'' has subsequently been adopted in the literature; see, e.g., Yue, Zhou, and So~\cite{YueZhouSo2019}.

\begin{definition}[Luo--Tseng error bound for the primal problem]
	\label{def:primal-prox-gradient-eb}
	We say that the Luo--Tseng error bound holds for the primal problem~\eqref{eq:P} if, for every \(\zeta\geq F^\star\), there exist constants \(\kappa_P>0\) and \(\epsilon_P>0\) such that
	\begin{equation}
		\label{eq:primal-prox-gradient-eb}
		\operatorname{dist}(v,\mathcal V)
		\leq
		\kappa_P\| r_P(v)\|,
		\qquad
		\forall v\in\mathbb R^n
		\text{ satisfying }
		F(v)\leq\zeta,\quad
		\| r_P(v)\|\leq\epsilon_P.
	\end{equation}
\end{definition}

Motivated by the preceding primal formulation, we use the same sublevel-set and small-residual structure to define a dual analogue, with the primal proximal-gradient residual replaced by the dual subdifferential residual \(r_D\). The choice of \(r_D\) is motivated by the subdifferential error bounds studied by Drusvyatskiy and Lewis~\cite{DrusvyatskiyLewis2018}. For brevity, we refer to this property as the dual Luo--Tseng error bound.

\begin{definition}[Luo--Tseng error bound for the dual problem] \label{def:dual-subdifferential-eb} We say that the Luo--Tseng error bound holds for the dual problem~\eqref{eq:Dmin} if, for every \(\eta\geq d^\star\), there exist constants \(\kappa_D>0\) and \(\epsilon_D>0\) such that
	\begin{equation}
		\label{eq:dual-subdifferential-eb}
		\operatorname{dist}(x,\Lambda)
		\leq
		\kappa_D r_D(x),
		\qquad
		\forall x\in\mathbb R^m
		\text{ satisfying }
		d(x)\leq\eta,\quad
		r_D(x)\leq\epsilon_D.
	\end{equation}
\end{definition}

The two notions differ mainly in their localization. A local error bound is centered at a fixed solution and holds only in a neighborhood of that solution, with constants that may depend on the solution. By contrast, a Luo--Tseng error bound is uniform on a prescribed objective sublevel set, subject to a small-residual condition, with constants depending on the chosen sublevel. Thus, the former is solution-centered, whereas the latter is sublevel-set based. We next examine the relations between these two error-bound properties.

\subsection{Relations between Luo--Tseng and   local error bounds}\label{subsec:sublevel-local-eb-relations}

The Luo--Tseng error bounds are uniform on prescribed objective sublevel sets, subject to a small-residual condition, whereas the   local error bounds are neighborhood properties at individual solutions. For the converse passage from the local error bounds to the Luo--Tseng error bounds, we use level-boundedness, recalled next.

\begin{definition}[Level-boundedness, {\cite[Definition~1.8]{RockafellarWets1998}}]\label{def:level-boundedness}
An extended-real-valued function $h:\mathbb R^p\to(-\infty,+\infty]$ is said to be level-bounded if, for every
\(\alpha\in\mathbb R\), its lower level set
\[
    \operatorname{lev}_{\leq\alpha}h
    :=
    \{z\in\mathbb R^p:h(z)\leq\alpha\}
\]
is bounded, possibly empty.
\end{definition}

Since the primal and dual objectives \(F\) and \(d\) are closed, their lower level sets are closed. Consequently, whenever \(F\) or \(d\) is level-bounded, each of its lower level sets at a finite level is compact.

The following proposition summarizes the relations between the two types of error bounds.

\begin{proposition}[Relations between local  and Luo--Tseng error bounds]	\label{prop:local-sublevel-relations}
	The following statements hold.
	
	\noindent \textnormal{(i)}
	Let \(\bar v\in\mathcal V\), and suppose that
	\(\nabla f_1\) is locally Lipschitz continuous around \(\bar v\).
	If the primal Luo--Tseng error bound holds, then the primal
	  local error bound holds at \(\bar v\).
	
	\noindent \textnormal{(ii)}
	If the dual Luo--Tseng error bound holds, then the dual
	  local error bound holds at every
	\(\bar x\in\Lambda\).
	
	\noindent \textnormal{(iii)}
	If \(F\) is level-bounded and the primal   local
	error bound holds at every \(\bar v\in\mathcal V\), then the
	primal Luo--Tseng error bound holds.
	
	\noindent \textnormal{(iv)}
	If \(d\) is level-bounded and the dual   local
	error bound holds at every \(\bar x\in\Lambda\), then the
	dual Luo--Tseng error bound holds.
\end{proposition}

\begin{proof}
	See Appendix~\ref{app:local-sublevel-relations}.
\end{proof}

%
%


We next present two simple structural conditions under which the level-boundedness required in the preceding results holds.

\begin{corollary}[Sufficient conditions for level-boundedness]
	\label{cor:level-boundedness-sufficient-conditions}
	Suppose that
	Assumption~\ref{ass:basic-primal-dual-setting} holds.
	Then the following statements hold.
	
	\begin{enumerate}[label=\textnormal{(\roman*)},leftmargin=*]
		
		\item If \(f_1\) is strongly convex, then the primal objective \(F\) is level-bounded.
		
		\item If \(A\) has full row rank, then the dual objective \(d\) is level-bounded.
		
	\end{enumerate}
\end{corollary}

\begin{proof}
	See Appendix~\ref{app:cor:level-boundedness-sufficient-conditions}.
\end{proof}

\section{KKT metric subregularity and primal-dual local error bounds}\label{sec:primal-dual-eb-relation}
In this section, we study the relation between KKT metric subregularity and the primal and dual   local error bounds. We first establish implications in both directions and then derive equivalence results under structural assumptions.

Throughout this section, Assumption~\ref{ass:basic-primal-dual-setting} 
is maintained, and \((\bar v,\bar x)\in\mathcal S\) denotes a fixed primal--dual KKT solution. Our main results are summarized in Figure \ref{fig:abstract-kkt-eb-relations}.

\begin{figure}[htbp]
	\centering
	\includegraphics[width=\linewidth]{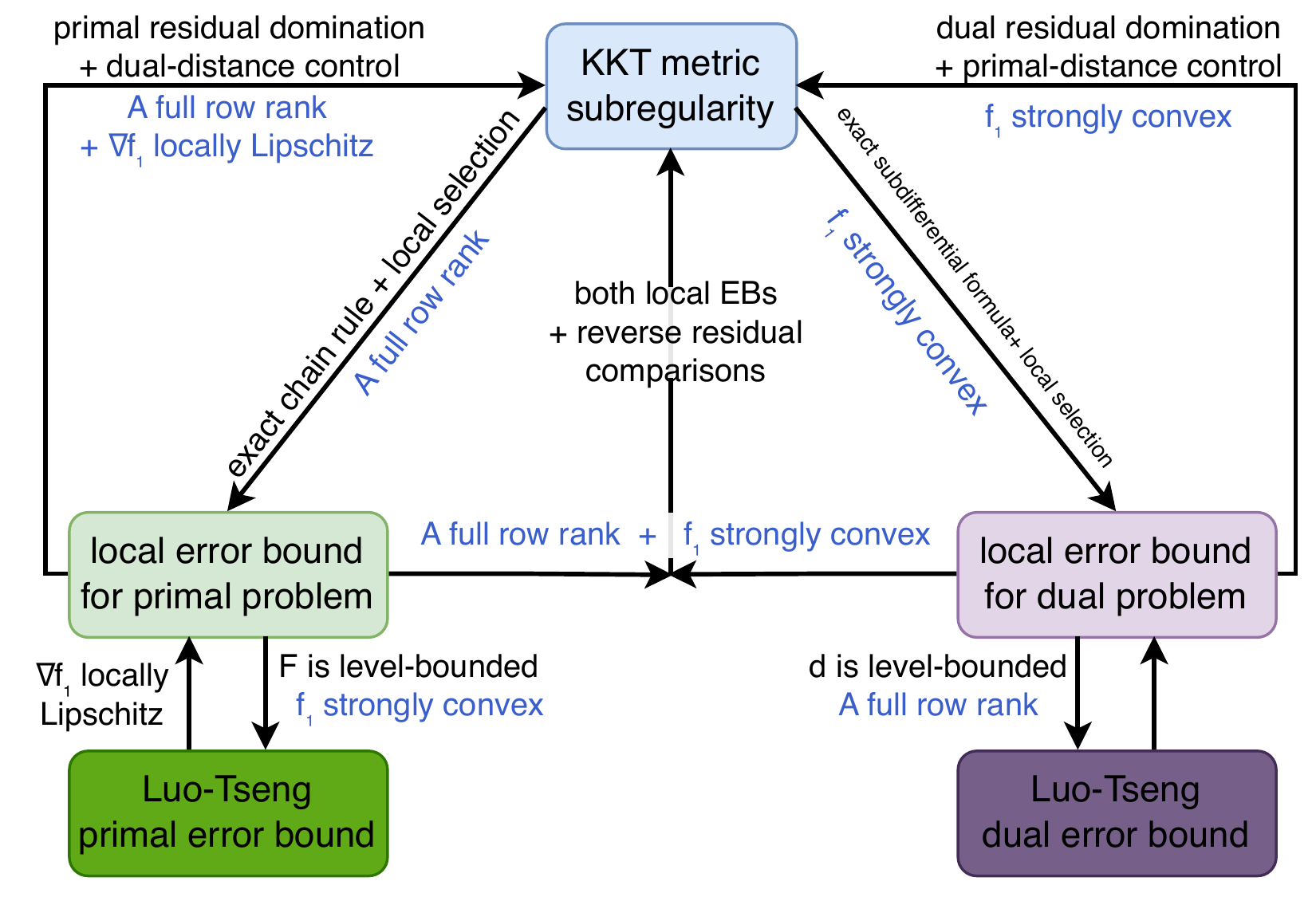}
	\caption{Implications among KKT metric subregularity,   local error bounds, and Luo--Tseng error bounds}
	\label{fig:abstract-kkt-eb-relations}
\end{figure}

\subsection{KKT metric subregularity implies   local error bounds}\label{subsec:kkt-msr-implies-eb}

We first study when KKT metric subregularity implies the primal and dual  local  error bounds.

We begin with the primal problem. The following lemma relates the primal residual to the KKT residual.
\begin{assumption}\label{ass:chainrule}
	 The exact subdifferential chain rule
	\[
	\partial(f_2\circ(A\cdot+b))(w) = A^\top\partial f_2(Aw+b)
	\]
	holds for all \(w\).
\end{assumption}

\begin{lemma}\label{lem:primal-residual-controls-kkt}
	Suppose that Assumption~\ref{ass:chainrule} holds.
	Then, for every \(v\in\mathbb R^n\), there exists
	\(x_P(v)\in\mathbb R^m\) such that
	\[
	r_{\rm KKT}(v,x_P(v))
	\le
	\sqrt{1+\|A\|^2}\,\|r_P(v)\|.
	\]
\end{lemma}
\begin{proof}
	See Appendix~\ref{app:proof:lem:primal-residual-controls-kkt}
\end{proof}

\begin{assumption}[Primal localization]
	\label{ass:primal-localization}
	There exists a selection $x_P:\mathbb R^n\to\mathbb R^m$ 	of the vectors furnished by
	Lemma~\ref{lem:primal-residual-controls-kkt} such that
	\(x_P(\bar v)=\bar x\) and \(x_P\) is continuous at
	\(\bar v\).
\end{assumption}

We next use this estimate to show that KKT metric subregularity implies
the primal local error bound.

\begin{theorem}\label{thm:kkt-msr-implies-primal-eb}
	Suppose that Assumptions~\ref{ass:chainrule} and
	\ref{ass:primal-localization} hold, and that the reduced KKT mapping
	\(\mathcal M\) is metrically subregular at
	\((\bar v,\bar x)\) for \(0\) as in
	Definition~\ref{def:kkt-msr-pointwise}, with constants
	\(\kappa_{\rm KKT}>0\) and \(\epsilon_{\rm KKT}>0\).
	Then there exists \(\rho_P>0\) such that
	\[
	\operatorname{dist}(v,\mathcal V)
	\leq
	\kappa_P\|r_P(v)\|,
	\qquad
	\forall v\in\mathbb B(\bar v,\rho_P),
	\]
	where $	\kappa_P	:=	\kappa_{\rm KKT}\sqrt{1+\|A\|^2}.$
\end{theorem}

\begin{proof}
	By Assumption~\ref{ass:primal-localization}, there exists
	\(\rho_P>0\) sufficiently small such that
	\[
	(v,x_P(v))
	\in
	\mathbb B((\bar v,\bar x),\epsilon_{\rm KKT}),
	\qquad
	\forall v\in\mathbb B(\bar v,\rho_P).
	\]
	Hence, for every such \(v\),
	\[
	\begin{aligned}
		\operatorname{dist}(v,\mathcal V)
		&\leq
		\operatorname{dist}((v,x_P(v)),\mathcal S)\\
		&\leq
		\kappa_{\rm KKT}
		r_{\rm KKT}(v,x_P(v))\\
		&\leq
		\kappa_{\rm KKT}\sqrt{1+\|A\|^2}\,
		\|r_P(v)\|.
	\end{aligned}
	\]
	The conclusion follows.
\end{proof}

Theorem~\ref{thm:kkt-msr-implies-primal-eb} is stated under abstract chain-rule and localization conditions. The following corollary shows that both conditions are automatically satisfied when \(A\) has full row rank.

\begin{corollary}
	\label{cor:full-row-rank-kkt-msr-implies-primal-eb}
	Assume that \(A\) has full row rank. If the reduced KKT mapping
	\(\mathcal M\) is metrically subregular at
	\((\bar v,\bar x)\) for \(0\), then the local error bound holds at
	\(\bar v\) for the primal problem~\eqref{eq:P}.
\end{corollary}

\begin{proof}
	See Appendix~\ref{app:proof:cor:full-row-rank-kkt-msr-implies-primal-eb}.
\end{proof}

We next consider the dual problem. The following assumption gives the
subdifferential formula needed to relate the dual residual to the KKT
residual.

\begin{assumption}[Exact dual subdifferential formula]
	\label{ass:dual-subdifferential}
	The exact subdifferential formula
	\[
	\partial d(x)
	=
	\left\{
	u-Av-b:
	u\in\partial f_2^*(x),\
	v\in\partial f_1^*(-A^\top x)
	\right\}
	\]
	holds for every \(x\in\mathbb R^m\).
\end{assumption}

\begin{lemma}\label{lem:dual-residual-controls-kkt}
	Suppose that Assumption~\ref{ass:dual-subdifferential} holds.
	Then, for every \(x\in\mathbb R^m\) with \(r_D(x)<+\infty\),
	there exists \(v_D(x)\in\mathbb R^n\) such that
	\[
	r_{\rm KKT}(v_D(x),x)
	\leq
	r_D(x).
	\]
\end{lemma}

\begin{proof}
	See Appendix~\ref{app:proof:lem:dual-residual-controls-kkt}.
\end{proof}

\begin{assumption}[Dual localization]
	\label{ass:dual-localization}
	There exists a selection $	v_D:\operatorname{dom}\partial d\to\mathbb R^n$ 	of the vectors furnished by 	Lemma~\ref{lem:dual-residual-controls-kkt} such that
	\(v_D(\bar x)=\bar v\) and \(v_D\) is continuous at
	\(\bar x\) relative to \(\operatorname{dom}\partial d\).
\end{assumption}

We next use this estimate to show that KKT metric subregularity implies
the dual local error bound.

\begin{theorem}\label{thm:kkt-msr-implies-dual-eb}
	Suppose that Assumptions~\ref{ass:dual-subdifferential} and
	\ref{ass:dual-localization} hold, and that the reduced KKT mapping
	\(\mathcal M\) is metrically subregular at
	\((\bar v,\bar x)\) for \(0\) as in
	Definition~\ref{def:kkt-msr-pointwise}, with constants
	\(\kappa_{\rm KKT}>0\) and \(\epsilon_{\rm KKT}>0\).
	Then there exists \(\rho_D>0\) such that
	\[
	\operatorname{dist}(x,\Lambda)
	\leq
	\kappa_D r_D(x),
	\qquad
	\forall x\in\mathbb B(\bar x,\rho_D),
	\]
	where \(\kappa_D:=\kappa_{\rm KKT}.\)
\end{theorem}

\begin{proof}
	By Assumption~\ref{ass:dual-localization}, there exists
	\(\rho_D>0\) sufficiently small such that
	\begin{equation}\label{eq:dual-pair-localization}
		(v_D(x),x)
		\in
		\mathbb B((\bar v,\bar x),\epsilon_{\rm KKT}),
	\end{equation}
	for every \(x\in\mathbb B(\bar x,\rho_D)\) satisfying
	\(r_D(x)<+\infty\).
	
	Take any \(x\in\mathbb B(\bar x,\rho_D)\).
	If \(r_D(x)=+\infty\), then
	\[
	\operatorname{dist}(x,\Lambda)
	\leq
	\kappa_{\rm KKT} r_D(x)
	\]
	holds automatically.
	
	Suppose now that \(r_D(x)<+\infty\). Since
	\((v_D(x),x)\) lies in the neighborhood on which metric
	subregularity is valid, we have
	\[
	\operatorname{dist}((v_D(x),x),\mathcal S)
	\leq
	\kappa_{\rm KKT}
	r_{\rm KKT}(v_D(x),x).
	\]
	
	By Proposition~\ref{prop:kkt-solution-product},
	\(\mathcal S=\mathcal V\times\Lambda\). Therefore,
	\[
	\operatorname{dist}(x,\Lambda)
	\leq
	\operatorname{dist}((v_D(x),x),\mathcal S).
	\]
	
	Combining this inequality with
	Lemma~\ref{lem:dual-residual-controls-kkt}, we obtain
	\[
	\begin{aligned}
		\operatorname{dist}(x,\Lambda)
		&\leq
		\operatorname{dist}((v_D(x),x),\mathcal S) \leq
		\kappa_{\rm KKT}
		r_{\rm KKT}(v_D(x),x) \leq
		\kappa_{\rm KKT}r_D(x).
	\end{aligned}
	\]
	Thus, in both cases, it holds that
	\[
	\operatorname{dist}(x,\Lambda)
	\leq
	\kappa_{\rm KKT}r_D(x).
	\]
	Since \(x\in\mathbb B(\bar x,\rho_D)\) was arbitrary, the local
	error bound holds at \(\bar x\) for the dual problem, with
	\(\kappa_D=\kappa_{\rm KKT}\).
\end{proof}

Theorem~\ref{thm:kkt-msr-implies-dual-eb} is stated under abstract
subdifferential and localization conditions. The following corollary
shows that both conditions are automatically satisfied when \(f_1\)
is strongly convex.

\begin{corollary}\label{cor:strong-convexity-kkt-msr-implies-dual-eb}
	Assume that \(f_1\) is \(\sigma\)-strongly convex for some
	\(\sigma>0\). If the reduced KKT mapping \(\mathcal M\) is metrically
	subregular at \((\bar v,\bar x)\) for \(0\), then the local error bound
	holds at \(\bar x\) for the dual problem~\eqref{eq:Dmin}.
\end{corollary}

\begin{proof}
	See Appendix~\ref{app:proof:cor:strong-convexity-kkt-msr-implies-dual-eb}.
\end{proof}

\subsection{Local error bounds imply KKT metric subregularity}
\label{subsec:eb-implies-kkt-msr}
We now consider the converse direction. We first show that the primal and dual   local error bounds together imply KKT metric subregularity under suitable residual comparisons.

\begin{theorem}	\label{thm:primal-dual-eb-implies-kkt-msr}
	Assume that the   local error bounds hold at 	\(\bar v\) and \(\bar x\) for the primal and dual problems, respectively, 	with constants \((\rho_P,\kappa_P)\) and \((\rho_D,\kappa_D)\).
	
	Suppose further that there exist constants 	\(\rho_R>0\) and \(c_P,c_D>0\) such that
	\begin{equation}
		\label{eq:reverse-residual-comparisons}
		\| r_P(v)\|
		\leq
		c_P r_{\rm KKT}(v,x),
		\qquad
		r_D(x)
		\leq
		c_D r_{\rm KKT}(v,x)
	\end{equation}
	for every 	\((v,x)\in\mathbb B((\bar v,\bar x),\rho_R)\).
	
	Then the reduced KKT mapping \(\mathcal M\) is metrically subregular at 	\((\bar v,\bar x)\) for \(0\). More precisely, with $	\rho_0:=\min\{\rho_P,\rho_D,\rho_R\},$ 	one has
	\[
	\operatorname{dist}((v,x),\mathcal S)
	\leq
	\kappa_{\rm KKT}r_{\rm KKT}(v,x),
	\qquad
	\forall (v,x)\in
	\mathbb B((\bar v,\bar x),\rho_0),
	\]
	where 	$\kappa_{\rm KKT}	:=	\sqrt{\kappa_P^2c_P^2+\kappa_D^2c_D^2}.	$
\end{theorem}

\begin{proof}
	See Appendix~\ref{app:proof:thm:primal-dual-eb-implies-kkt-msr}
\end{proof}

The residual comparisons in the preceding theorem are abstract. The following proposition provides concrete conditions under which these comparisons hold.
\begin{proposition}[Residual comparisons]
	\label{prop:structural-residual-comparisons}
	The following statements hold.
	
	\textnormal{(i)}
	If \(A\) has full row rank, then
	\[
	\|r_P(v)\|
	\leq
	\sqrt{1+4\|A^\dagger\|^2}\,
	r_{\rm KKT}(v,x)
	\]
	for every \((v,x)\in\mathbb R^n\times\mathbb R^m\).
	
	\textnormal{(ii)}
	If \(f_1\) is \(\sigma\)-strongly convex for some \(\sigma>0\), then
	\[
	r_D(x)
	\leq
	\sqrt{1+\frac{\|A\|^2}{\sigma^2}}\,
	r_{\rm KKT}(v,x)
	\]
	for every \((v,x)\in\mathbb R^n\times\mathbb R^m\).
\end{proposition}
\begin{proof}
	See Appendix~\ref{app:lem:full-row-rank-primal-residual-comparison}.
\end{proof}
Thus, when \(A\) has full row rank and \(f_1\) is strongly convex, the residual-comparison conditions in Theorem~\ref{thm:primal-dual-eb-implies-kkt-msr} hold globally.
%
%
%
%
%

The preceding theorem uses both the primal and dual   local error bounds. We next consider whether one   local error bound alone is sufficient.

\begin{theorem} 	\label{thm:primal-eb-implies-kkt-msr}	
	Assume that the primal   local error bound holds at 	\(\bar v\), with constants \(\rho_P>0\) and \(\kappa_P>0\).
	
	Suppose further that there exist constants	\(\rho_1,c_P,\alpha,\beta>0\) such that
	\begin{equation}
		\label{eq:primal-residual-kkt-domination}	\| r_P(v)\|	\leq	c_P r_{\rm KKT}(v,x)
	\end{equation}
	and
	\begin{equation}
		\label{eq:dual-distance-primal-control}
		\operatorname{dist}(x,\Lambda)
		\leq
		\alpha r_{\rm KKT}(v,x)
		+
		\beta\operatorname{dist}(v,\mathcal V)
	\end{equation}
	for every
	\((v,x)\in\mathbb B((\bar v,\bar x),\rho_1)\).
	
	Then the reduced KKT mapping \(\mathcal M\) is metrically subregular at
	\((\bar v,\bar x)\) for \(0\). More precisely, with $\rho:=\min\{\rho_P,\rho_1\},$
	one has
	\[
	\operatorname{dist}((v,x),\mathcal S)
	\leq
	\kappa_{\rm KKT}r_{\rm KKT}(v,x),
	\qquad
	\forall (v,x)\in
	\mathbb B((\bar v,\bar x),\rho),
	\]
	where $	\kappa_{\rm KKT}	:=	\sqrt{
		\kappa_P^2c_P^2
		+
		\left(
		\alpha+\beta\kappa_Pc_P
		\right)^2
	}.$
\end{theorem}

\begin{proof}
	See Appendix~\ref{app:proof-primal-eb-implies-kkt-msr}.
\end{proof}

Theorem  \ref{thm:primal-eb-implies-kkt-msr} is stated under abstract residual and distance-control conditions. The following corollary gives a concrete sufficient condition.
\begin{corollary}
\label{cor:primal-eb-structural-kkt-msr}

Assume that the   local error bound holds at \(\bar v\) for the primal problem \eqref{eq:P}. Assume further that \(A\) has full row rank and that \(\nabla f_1\) is locally Lipschitz continuous around \(\bar v\).
Then the reduced KKT mapping \(\mathcal M\) is metrically subregular at \((\bar v,\bar x)\) for \(0\).
\end{corollary}

\begin{proof}
Since \(A\) has full row rank, \(A^\top\) is injective. Moreover, Proposition~\ref{prop:kkt-solution-product} gives
\(\mathcal S=\mathcal V\times\Lambda\). If \(x_1,x_2\in\Lambda\), then
\[
    \nabla f_1(\bar v)+A^\top x_1=0,
    \qquad
    \nabla f_1(\bar v)+A^\top x_2=0.
\]
Thus \(A^\top(x_1-x_2)=0\), and the injectivity of \(A^\top\) implies \(x_1=x_2\). Hence $ \Lambda=\{\bar x\}.$

Since \(F\) is closed and convex, \(\mathcal V\) is nonempty, closed, and convex. Hence the Euclidean projection onto \(\mathcal V\) is single-valued and nonexpansive. For every \(v\) near \(\bar v\), let \(v^\Pi\) be the Euclidean projection of \(v\) onto \(\mathcal V\). Since \(\bar v\in\mathcal V\), we have
\[
    \|v^\Pi-\bar v\|
    \leq
    \|v-\bar v\|.
\]
Let \(L>0\) be a local Lipschitz constant of \(\nabla f_1\) around \(\bar v\). After possibly shrinking the neighborhood, both \(v\) and \(v^\Pi\) lie in a neighborhood on which \(\nabla f_1\) is \(L\)-Lipschitz continuous.

Since \(v^\Pi\in\mathcal V\), \(\bar x\in\Lambda\), and \(\mathcal S=\mathcal V\times\Lambda\), we have
\[
    \nabla f_1(v^\Pi)+A^\top\bar x=0.
\]
Since $(A^\dagger)^\top A^\top=I_m,$ we have
\[
\begin{aligned}
\operatorname{dist}(x,\Lambda)
&=
\|x-\bar x\|\\
&\leq
\|A^\dagger\|
\|A^\top(x-\bar x)\|\\
&\leq
\|A^\dagger\|
\left(
    \|\nabla f_1(v)+A^\top x\|
    +
    \|\nabla f_1(v)-\nabla f_1(v^\Pi)\|
\right)\\
&\leq
\|A^\dagger\|r_{\rm KKT}(v,x)
+
L\|A^\dagger\|\operatorname{dist}(v,\mathcal V).
\end{aligned}
\]
Thus \eqref{eq:dual-distance-primal-control} holds locally with $\alpha=\|A^\dagger\|, \ \beta=L\|A^\dagger\|.$
    
By Proposition~\ref{prop:structural-residual-comparisons} \textnormal{(i)},
\[
    \| r_P(v)\|
    \leq
    \sqrt{1+4\|A^\dagger\|^2}\,
    r_{\rm KKT}(v,x),
\]
so that \eqref{eq:primal-residual-kkt-domination} holds with $c_P=\sqrt{1+4\|A^\dagger\|^2}.$
    
The conclusion now follows from Theorem~\ref{thm:primal-eb-implies-kkt-msr}.
\end{proof}

We next give the corresponding result for the dual   local error bound.
\begin{theorem}	\label{thm:dual-eb-implies-kkt-msr}	
	Assume that the   local error bound holds at	\(\bar x\) for the dual problem, with constants	\((\rho_D,\kappa_D)\).	Suppose further that there exist constants	\(\rho_1,c_D,\alpha,\beta>0\) such that
	\begin{equation}
		\label{eq:dual-residual-kkt-domination}
		r_D(x)
		\leq
		c_D r_{\rm KKT}(v,x)
	\end{equation}
	and
	\begin{equation}
		\label{eq:primal-distance-dual-control}
		\operatorname{dist}(v,\mathcal V)
		\leq
		\alpha r_{\rm KKT}(v,x)
		+
		\beta\operatorname{dist}(x,\Lambda)
	\end{equation}
	for every
	\((v,x)\in\mathbb B((\bar v,\bar x),\rho_1)\).
	
	Then the reduced KKT mapping \(\mathcal M\) is metrically subregular at
	\((\bar v,\bar x)\) for \(0\). More precisely, with $	\rho:=\min\{\rho_D,\rho_1\},$
	one has
	\[
	\operatorname{dist}((v,x),\mathcal S)
	\leq
	\kappa_{\rm KKT}r_{\rm KKT}(v,x),
	\qquad
	\forall (v,x)\in
	\mathbb B((\bar v,\bar x),\rho),
	\]
	where $	\kappa_{\rm KKT}	:=
	\sqrt{
		\left(
		\alpha+\beta\kappa_Dc_D
		\right)^2
		+
		\kappa_D^2c_D^2	}.$
\end{theorem}

\begin{proof}
	See Appendix~\ref{app:proof:thm:dual-eb-implies-kkt-msr}.
\end{proof}

The following corollary gives a concrete sufficient condition for Theorem \ref{thm:dual-eb-implies-kkt-msr}.

\begin{corollary} \label{cor:dual-eb-structural-kkt-msr}

Assume that the   local error bound holds at \(\bar x\) for the dual problem \eqref{eq:Dmin}. Assume further that \(f_1\) is \(\sigma\)-strongly convex for some \(\sigma>0\). Then the reduced KKT mapping \(\mathcal M\) is metrically subregular at \((\bar v,\bar x)\) for \(0\).
\end{corollary}

\begin{proof}
Since \(f_1\) is \(\sigma\)-strongly convex and \(G=f_2\circ(A\cdot+b)\) is convex, \(F=f_1+G\) is \(\sigma\)-strongly convex. Hence the primal solution is unique, and $\mathcal V=\{\bar v\}.$

Since \(f_1\) is strongly convex, \(f_1^*\) is finite and continuously differentiable on \(\mathbb R^n\). It follows that \(d\) is closed and convex. Hence \(\Lambda\) is nonempty, closed, and convex. For every \(x\in\mathbb R^m\), let $x^\Pi:=\operatorname{proj}_{\Lambda}(x).$ Since \(\bar v\in\mathcal V\), \(x^\Pi\in\Lambda\), and \(\mathcal S=\mathcal V\times\Lambda\), we have
\[
    \nabla f_1(\bar v)+A^\top x^\Pi=0.
\]

By the strong monotonicity of \(\nabla f_1\),
\[
\begin{aligned}
\sigma\|v-\bar v\|^2
&\leq
\left\langle
\nabla f_1(v)-\nabla f_1(\bar v),v-\bar v
\right\rangle\\
&=
\left\langle
\nabla f_1(v)+A^\top x,v-\bar v
\right\rangle
-
\left\langle
A^\top(x-x^\Pi),v-\bar v
\right\rangle\\
&\leq
\left(
    \|\nabla f_1(v)+A^\top x\|
    +
    \|A\|\|x-x^\Pi\|
\right)
\|v-\bar v\|.
\end{aligned}
\]
If \(v=\bar v\), the following estimate holds trivially. Otherwise, dividing the preceding inequality by \(\|v-\bar v\|\) gives
\[
\begin{aligned}
\operatorname{dist}(v,\mathcal V)
&=
\|v-\bar v\|\\
&\leq
\frac{1}{\sigma}
\|\nabla f_1(v)+A^\top x\|
+
\frac{\|A\|}{\sigma}
\operatorname{dist}(x,\Lambda)\\
&\leq
\frac{1}{\sigma}r_{\rm KKT}(v,x)
+
\frac{\|A\|}{\sigma}
\operatorname{dist}(x,\Lambda).
\end{aligned}
\]
Thus \eqref{eq:primal-distance-dual-control} holds with$
    \alpha=\frac{1}{\sigma},
    \ 
    \beta=\frac{\|A\|}{\sigma}.$

By Proposition~\ref{prop:structural-residual-comparisons} \textnormal{(ii)},
\eqref{eq:dual-residual-kkt-domination} holds with $ c_D  =    \sqrt{1+\frac{\|A\|^2}{\sigma^2}}.$
The conclusion now follows from Theorem~\ref{thm:dual-eb-implies-kkt-msr}.
\end{proof}

\subsection{Equivalence results for KKT metric subregularity and   local error bounds}\label{subsec:kkt-eb-equivalence}

Combining the implication results above yields the following abstract local equivalence.

\begin{corollary}\label{cor:abstract-local-equivalence}
	Suppose that Assumptions~\ref{ass:chainrule}, 	\ref{ass:primal-localization}, 	\ref{ass:dual-subdifferential}, and 	\ref{ass:dual-localization} hold. 	Assume further that the reverse residual comparisons \eqref{eq:reverse-residual-comparisons} hold locally around
	\((\bar v,\bar x)\).
	Then the following statements are equivalent:
	\begin{enumerate}
		\item The reduced KKT mapping \(\mathcal M\) is metrically
		subregular at \((\bar v,\bar x)\) for \(0\).
		
		\item The local error bound holds at \(\bar v\) for the primal
		problem~\eqref{eq:P}, and the local error bound holds at
		\(\bar x\) for the dual problem~\eqref{eq:Dmin}.
	\end{enumerate}
\end{corollary}

\begin{proof}
	If \(\mathcal M\) is metrically subregular at 	\((\bar v,\bar x)\) for \(0\), then 	Theorems~\ref{thm:kkt-msr-implies-primal-eb} and 	\ref{thm:kkt-msr-implies-dual-eb} yield the primal and dual local
	error bounds, respectively.
	
	Conversely, if the primal and dual local error bounds hold, then 	Theorem~\ref{thm:primal-dual-eb-implies-kkt-msr}, together with 	\eqref{eq:reverse-residual-comparisons}, yields metric subregularity 	of \(\mathcal M\) at \((\bar v,\bar x)\) for \(0\).
\end{proof}

We next derive more explicit equivalence results under structural conditions, beginning with the primal problem.
\begin{corollary}	\label{cor:primal-eb-kkt-msr-equivalence}
	Assume that \(A\) has full row rank and that \(\nabla f_1\) is locally Lipschitz continuous around \(\bar v\). Then the following statements are equivalent:
	\begin{enumerate}
		\item The reduced KKT mapping \(\mathcal M\) is metrically subregular at \((\bar v,\bar x)\) for \(0\).
		\item The   local error bound holds at \(\bar v\) for the primal problem \eqref{eq:P}.
	\end{enumerate}
\end{corollary}

\begin{proof}
	The first statement implies the second by Corollary~\ref{cor:full-row-rank-kkt-msr-implies-primal-eb}.
	The second statement implies the first by Corollary~\ref{cor:primal-eb-structural-kkt-msr}.
\end{proof}

The corresponding dual equivalence is given next.
\begin{corollary}\label{cor:dual-eb-kkt-msr-equivalence}
 Assume that \(f_1\) is \(\sigma\)-strongly convex for some \(\sigma>0\). Then the following statements are equivalent:
\begin{enumerate}
\item The reduced KKT mapping \(\mathcal M\) is metrically subregular at \((\bar v,\bar x)\) for \(0\).
\item The   local error bound holds at \(\bar x\) for the dual problem \eqref{eq:Dmin}.
\end{enumerate}
\end{corollary}

\begin{proof}
The first statement implies the second by Corollary~\ref{cor:strong-convexity-kkt-msr-implies-dual-eb}.
The second statement implies the first by Corollary~\ref{cor:dual-eb-structural-kkt-msr}.
\end{proof}

Finally, combining the structural results on both sides yields a simultaneous primal--dual equivalence.
\begin{corollary}\label{cor:structural-local-equivalence}
	Assume that \(A\) has full row rank and that \(f_1\) is \(\sigma\)-strongly convex for some \(\sigma>0\).
	Then the following statements are equivalent:
	\begin{enumerate}
		\item The reduced KKT mapping \(\mathcal M\) is metrically subregular at \((\bar v,\bar x)\) for \(0\).
		
		\item The   local error bound holds at \(\bar v\) for the primal problem \eqref{eq:P}, and the   local error bound holds at \(\bar x\) for the dual problem
		\eqref{eq:Dmin}.
	\end{enumerate}
\end{corollary}

\begin{proof}
	If \(\mathcal M\) is metrically subregular, then the primal local error bound follows from Corollary~\ref{cor:full-row-rank-kkt-msr-implies-primal-eb},	while the dual local error bound follows from	Corollary~\ref{cor:strong-convexity-kkt-msr-implies-dual-eb}.
	
	Conversely, the dual local error bound together with the strong	convexity of \(f_1\) implies metric subregularity of \(\mathcal M\)	by Corollary~\ref{cor:dual-eb-structural-kkt-msr}.
\end{proof}

\section{Applications}\label{sec:applications}
The purpose of this section is to illustrate how the preceding framework applies across several representative model classes. The applications play two complementary roles. For models for which primal error bounds are already available in the literature, our framework provides a common KKT-based route to recover these bounds and, at the same time, yields corresponding dual error bounds and KKT metric subregularity results. For other models or formulations not explicitly covered by the model-specific error-bound results discussed below, the same framework yields the relevant error bounds once the required structural conditions are verified.

For each model class, we first establish a general error-bound result.  We then apply this result to several representative examples. Any additional assumptions needed for the Luo--Tseng error bounds are stated explicitly.

\subsection{Convex quadratic models with polyhedral regularizers}
\label{subsec:quadratic-polyhedral-regularizers}

We first apply the preceding primal-dual framework to convex quadratic optimization with a polyhedral regularizer. Consider

\begin{equation}\label{eq:quadratic-polyhedral-primal}
    \min_{v\in\mathbb R^n}
    \left\{
        F_Q(v):=q(v)+\psi(v)
    \right\},
    \qquad
    q(v)
    :=
    \frac12\langle v,Qv\rangle
    +
    \langle c,v\rangle
    +
    \gamma,
    \tag{QP}
\end{equation}
where \(Q\in\mathbb R^{n\times n}\) is symmetric and positive semidefinite, \(\gamma\in\mathbb R\), \(c\in\mathbb R^n\), and \(\psi:\mathbb R^n\to(-\infty,+\infty]\) is a proper, closed, polyhedral convex function. Recall that a convex function is polyhedral if its epigraph is a polyhedral set. Problem~\eqref{eq:quadratic-polyhedral-primal} is a special case of \eqref{eq:P}. 
In the natural composite representation, \(A=I\); hence \(A\) has full row rank. Moreover, \(\nabla q(v)=Qv+c\) is globally Lipschitz continuous with Lipschitz constant \(\|Q\|\).

The Fenchel dual problem is
\begin{equation}
\label{eq:quadratic-polyhedral-dual}
    \min_{x\in\mathbb R^n}
    d_Q(x)
    :=
    \psi^*(x)+q^*(-x).
    \tag{QD}
\end{equation}
Since
\[
    q^*(y) =
    \frac12
    \left\langle
        y-c,Q^\dagger(y-c)
    \right\rangle
    - \gamma +
    \delta_{\operatorname{range}(Q)}(y-c),
\]
the dual objective can equivalently be written as
\begin{equation}
\label{eq:quadratic-polyhedral-dual-explicit}
    d_Q(x)
    =
    \psi^*(x)
    +
    \frac12
    \left\langle
        x+c,Q^\dagger(x+c)
    \right\rangle
    -
    \gamma
    +
    \delta_{\operatorname{range}(Q)}(x+c).
\end{equation}

Let $ \mathcal V_Q:=\arg\min F_Q, \  \Lambda_Q:=\arg\min d_Q.$

The associated reduced KKT mapping is
\begin{equation}
\label{eq:quadratic-polyhedral-kkt-mapping}
    \mathcal M_Q(v,x)
    :=
    \begin{pmatrix}
        Qv+c+x\\[1mm]
        v-\partial\psi^*(x)
    \end{pmatrix}.
\end{equation}

\begin{lemma}
\label{lem:quadratic-polyhedral-kkt}
The mapping \(\mathcal M_Q\) in \eqref{eq:quadratic-polyhedral-kkt-mapping} is a piecewise polyhedral multifunction. Consequently, it is metrically subregular at every \((\bar v,\bar x)\in \mathcal M_Q^{-1}(0)\) for \(0\).
\end{lemma}

\begin{proof}
Since \(\psi\) is polyhedral convex, its Fenchel conjugate \(\psi^*\) is also polyhedral convex; see \cite[Theorem~19.2]{Rockafellar1970}. Moreover, \(\partial\psi^*\) is a piecewise polyhedral multifunction.

The graph of \(\mathcal M_Q\) admits the representation
\[
\begin{aligned}
    \operatorname{gph}\mathcal M_Q
    =
    \bigl\{
        ((v,x),(a,e)):
        \ &a=Qv+c+x,\ (x,v-e)\in
        \operatorname{gph}\partial\psi^*
    \bigr\}.
\end{aligned}
\]
It is therefore the inverse image of a finite union of polyhedral sets under an affine mapping. Hence \(\mathcal M_Q\) is piecewise polyhedral.

By Robinson~\cite[Proposition~1 and Corollary]{Robinson2009}, every piecewise polyhedral multifunction is metrically subregular at each point of its graph. Thus, for every \((\bar v,\bar x)\in\mathcal M_Q^{-1}(0)\), there exist constants
\(\kappa_Q,\rho_Q>0\) such that
\[
    \operatorname{dist}
    \bigl((v,x),\mathcal M_Q^{-1}(0)\bigr)
    \leq
    \kappa_Q
    \operatorname{dist}
    \bigl(0,\mathcal M_Q(v,x)\bigr)
\]
whenever $(v,x)\in \mathbb B\bigl((\bar v,\bar x),\rho_Q\bigr).$
\end{proof}

Since \(q\) is finite-valued and continuous on \(\mathbb R^n\), \(F_Q=q+\psi\) is proper, closed, and convex. Moreover, if problem~\eqref{eq:quadratic-polyhedral-primal} has an optimal solution, then the convex sum rule and the Fenchel--Young equality imply that the dual problem~\eqref{eq:quadratic-polyhedral-dual} has an optimal solution and that strong duality holds. Hence Assumption~\ref{ass:basic-primal-dual-setting} is satisfied.
\begin{proposition}[Error bounds for quadratic-polyhedral models]
	\label{prop:quadratic-polyhedral-error-bounds}
	Suppose that problem~\eqref{eq:quadratic-polyhedral-primal} has an optimal solution. Then:
	
	\begin{enumerate}[label=\textnormal{(\roman*)},leftmargin=*]
		\item The primal   local error bound holds at every
		\(\bar v\in\mathcal V_Q\).
		
		\item The dual   local error bound holds at every
		\(\bar x\in\Lambda_Q\).
		
		\item If \(F_Q\) is level-bounded, then the primal Luo--Tseng
		error bound holds.
		
		\item The dual Luo--Tseng error bound holds.
		
		\item If \(Q\succ0\), then \(F_Q\) is level-bounded and hence
		all four error-bound properties above hold.
	\end{enumerate}
\end{proposition}
\begin{proof}
	See Appendix~\ref{app:prop:quadratic-polyhedral-error-bounds}.
\end{proof}

\begin{remark}[Relation with existing primal error-bound results]
\label{rem:quadratic-polyhedral-literature}
Quadratic objectives with polyhedral regularizers form a standard class in the primal error-bound literature. Tseng and
Yun~\cite{TsengYun2009} establish their proximal-residual error bound when the smooth term is quadratic and the regularizer is polyhedral. Zhou and So~\cite{zhouso2017} obtain related primal error bounds for structured convex problems with polyhedral regularizers through metric subregularity of the regularizer subdifferential and a Hoffman-type bounded-linear-regularity argument. These general results apply to several of the models below after identifying the corresponding polyhedral regularizer; they do not necessarily discuss each named model separately.

Specialized error-bound results for support vector machines are studied in \cite{DeLeoneLazzari2010}. The approach used here differs in that the complete reduced KKT mapping is incorporated into a common primal--dual framework, yielding the primal and dual local error bounds together with the Luo--Tseng conclusions in
Proposition~\ref{prop:quadratic-polyhedral-error-bounds}.
\end{remark}

\begin{example}[Lasso-type models and SLOPE] 	\label{ex:lasso-type-models}
	
	Consider the regularized least-squares problem
	\begin{equation}
		\label{eq:lasso-type-models}
		\min_{\beta\in\mathbb R^p}
		\left\{
		\frac12\|y-X\beta\|^2
		+
		\lambda\phi(\beta)
		\right\},
		\qquad
		\lambda>0,
	\end{equation}
	where \(X\in\mathbb R^{m\times p}\), \(y\in\mathbb R^m\), and 	\(\phi:\mathbb R^p\to\mathbb R_+\) is a convex regularizer. 	Several standard models are obtained from different choices of
	\(\phi\):
	\[
	\phi(\beta)
	=
	\begin{cases}
		\|\beta\|_1,
		& \text{Lasso~\cite{Osborne2000}},\\[1mm]
		\displaystyle
		\sum_{j=1}^p w_j|\beta_j|,
		\quad w_j>0,
		& \text{Weighted/Adaptive Lasso~\cite{Zou2006}},\\[3mm]
		\|D\beta\|_1,
		& \text{Generalized Lasso~\cite{Tibshirani2011}},\\[2mm]
		\displaystyle
		J_w(\beta)
		:=
		\sum_{i=1}^p
		w_i|\beta|_{(i)},\ \ 	w_1\geq\cdots\geq w_p\geq0, 	\  \ 
		w_1>0,
		& \text{SLOPE~\cite{Bogdan2015}},
	\end{cases}
	\]
	where \(D\in\mathbb R^{r\times p}\) is a prescribed penalty matrix, 	and, for SLOPE, \( 	|\beta|_{(1)} 	\geq 	\cdots 	\geq
	|\beta|_{(p)}\) denoting the ordered absolute values of the components of \(\beta\). In the adaptive Lasso, the weights \(w_j\) are chosen adaptively from a pilot estimator. 
	
	All these models are treated here from the primal side using the
	identity composite representation
	\[
	f_1(\beta)
	=
	\frac12\|y-X\beta\|^2,
	\qquad
	f_2(\beta)
	=
	\lambda\phi(\beta),
	\qquad
	A=I,
	\qquad
	b=0.
	\]
	The function \(f_1\) is finite-valued and convex quadratic, with
	\[
	\nabla f_1(\beta)
	=
	X^\top(X\beta-y),
	\]
	and hence has a globally Lipschitz continuous gradient. For each of
	the choices above, \(f_2\) is a proper, closed, polyhedral convex
	function. Therefore all these models belong to the
	quadratic-polyhedral class of
	Proposition~\ref{prop:quadratic-polyhedral-error-bounds}.
	
	For the Lasso, weighted Lasso, adaptive Lasso, and SLOPE, the
	regularizer \(\phi\) is coercive. Hence the primal objective is
	level-bounded and attains its minimum. Consequently,
	Proposition~\ref{prop:quadratic-polyhedral-error-bounds} yields both
	the primal and dual local error bounds and both Luo--Tseng error
	bounds. No full-column-rank assumption on \(X\) is required.
	
	For the generalized Lasso, \(\phi(\beta)=\|D\beta\|_1\) need not be
	coercive. Nevertheless, the primal problem admits an optimal
	solution. Indeed, its objective depends on \(\beta\) only through
	\((X\beta,D\beta)\), and the function
	\[
	(u,v)
	\mapsto
	\frac12\|y-u\|^2+\lambda\|v\|_1
	\]
	is coercive on the finite-dimensional subspace
	\[
	\{(X\beta,D\beta):\beta\in\mathbb R^p\}.
	\]
	Thus Proposition~\ref{prop:quadratic-polyhedral-error-bounds} gives
	the primal and dual local error bounds and the dual Luo--Tseng error
	bound. Moreover, if
	\[
	\ker X\cap\ker D=\{0\},
	\]
	then the primal objective is level-bounded, and hence the primal
	Luo--Tseng error bound also holds.
\end{example}

\begin{example}[Support vector machine models]
	\label{ex:svm-models}
	
	Support vector machines provide a broad family of classification
	and regression models whose optimization structure depends on the
	choice of loss function. We consider first support vector
	classification (SVC). Given training data
	\[
	(z_i,y_i)\in\mathbb R^d\times\{-1,1\},
	\qquad i=1,\ldots,m,
	\]
	a broad class of soft-margin SVC models can be written as
	\begin{equation}
		\label{eq:general-svc}
		\min_{\omega\in\mathbb R^d,\;\beta\in\mathbb R}
		\left\{
		\frac12\|\omega\|^2
		+
		C\sum_{i=1}^m
		\ell\!\left(
		1-y_i(\omega^\top z_i+\beta)
		\right)
		\right\},
		\qquad C>0.
	\end{equation}
	Several convex loss functions in \eqref{eq:general-svc} yield	problems of the form \eqref{eq:quadratic-polyhedral-primal},	either directly or after an equivalent quadratic-programming	reformulation.
	
\textnormal{(a) \(L_1\)-loss.}
The standard \(L_1\)-loss SVC \cite{CortesVapnik1995,ChangHsiehLin2008}  uses
\[
\ell_{L_1}(t)
=
(t)_+
=
\max\{t,0\}.
\]

\textnormal{(b) \(L_2\)-loss.}
The standard \(L_2\)-loss SVC \cite{ChangHsiehLin2008} uses
\[
\ell_{L_2}(t)
=
(t)_+^2
=
\bigl(\max\{t,0\}\bigr)^2.
\]

\textnormal{(c) Pinball loss.}
	A pinball-loss SVC considered in 	\cite{JumutcHuangSuykens2013,HuangShiSuykens2014} uses
	\[
	\ell_{\rm p}^{\tau}(t)
	=
	\begin{cases}
		t, & t>0,\\
		-\tau t, & t\leq0,
	\end{cases}
	\qquad
	\tau\in(0,1].
	\]
	
	\textnormal{(d) \(\epsilon\)-insensitive pinball loss.}
	Huang, Shi, and Suykens~\cite{HuangShiSuykens2014} proposed
	\[
	\ell_{\rm ip}^{\tau,\epsilon}(t)
	=
	\begin{cases}
		t-\epsilon,
		& t>\epsilon,\\
		0,
		& -\epsilon/\tau\leq t\leq\epsilon,\\
		-\tau
		\left(t+\dfrac{\epsilon}{\tau}\right),
		& t<-\epsilon/\tau,
	\end{cases}
	\]
	where \(\tau\in(0,1]\) and \(\epsilon>0\).
	
	\textnormal{(e) Least-squares loss.}
	The least-squares SVC of Suykens and Vandewalle~\cite{SuykensVandewalle1999} uses
	\[
	\ell_{\rm ls}(t)=t^2.
	\]
	
%
	
	\textnormal{(f) Huber-pinball loss.}
	Zhu, Song, and Xiao~\cite{ZhuSongXiao2020} proposed
	\[
	\ell_{\rm hp}^{\tau,\delta}(t)
	=
	\begin{cases}
		t-\dfrac{\delta}{2}, & t>\delta,\\[1mm]
		\dfrac{t^2}{2\delta}, & 0<t\leq\delta,\\[1mm]
		\dfrac{\tau t^2}{2\delta}, & -\delta\leq t\leq0,\\[1mm]
		-\tau
		\left(t+\dfrac{\delta}{2}\right), & t<-\delta,
	\end{cases}
	\]
	where \(\tau\in(0,1]\) and \(\delta>0\).

	For support vector regression (SVR), given
	\[
	(z_i,y_i)\in\mathbb R^d\times\mathbb R,
	\qquad i=1,\ldots,m,
	\]
	consider
	\begin{equation}
		\label{eq:general-svr}
		\min_{\omega\in\mathbb R^d,\;\beta\in\mathbb R}
		\left\{
		\frac12\|\omega\|^2
		+
		C\sum_{i=1}^m
		L\!\left(
		y_i-\omega^\top z_i-\beta
		\right)
		\right\},
		\qquad C>0.
	\end{equation}
	
	\textnormal{(g) \(\epsilon\)-insensitive loss.}
	The classical \(\epsilon\)-SVR of	Drucker et al.~\cite{DruckerBurgesKaufmanSmolaVapnik1997}	uses
	\[
	L_\epsilon(t)
	=
	\max\{|t|-\epsilon,0\},
	\qquad
	\epsilon>0.
	\]
	
		\textnormal{(h) \(\epsilon\)-penalty loss.}
	Anand, Rastogi, and Chandra~\cite{AnandRastogiChandra2020}
	introduced
	\[
	L_{\tau_1,\tau_2,\epsilon}^{P}(t)
	=
	\max
	\left\{
	\tau_2|t|
	+
	(\tau_1-\tau_2)\epsilon,\,
	\tau_1|t|
	\right\},
	\]
	where $	\tau_2>0,	\ 	\tau_2\geq\tau_1\geq0,	\ 	\epsilon>0.$
	
	\textnormal{(i) Least-squares loss.}
	The least-squares SVR associated with	Suykens and Vandewalle~\cite{SuykensVandewalle1999} uses
	\[
	L_{\rm ls}(t)=t^2.
	\]
	
	\textnormal{(j) Huber loss.}
	The Huber-loss SVR discussed by	Gunn~\cite{Gunn1998} uses
	\[
	H_c(t)
	=
	\begin{cases}
		\dfrac12t^2, & |t|\leq c,\\[1mm]
		c\left(|t|-\dfrac c2\right), & |t|>c,
	\end{cases}
	\qquad c>0.
	\]
	
In addition to the loss-based formulations above, we also consider
the bias-regularized hinge-loss SVC
\begin{equation}
	\label{eq:svm-example}
	\min_{\omega\in\mathbb R^d,\;\beta\in\mathbb R}
	\left\{
	\frac12\|\omega\|^2
	+
	\frac12\beta^2
	+
	C\sum_{i=1}^m
	\max
	\left\{
	1-y_i(\omega^\top z_i+\beta),0
	\right\}
	\right\},
	\qquad C>0,
\end{equation}
where \((z_i,y_i)\in\mathbb R^d\times\{-1,1\},
\ \  i=1,\ldots,m.\)

Bias-regularized SVM formulations are discussed by Piccialli and Sciandrone~\cite{Piccialli2022}. The hinge-loss term is polyhedral, while the quadratic term in \((\omega,\beta)\) is strongly convex. Hence this model also belongs to the quadratic-polyhedral class.

\begin{proposition}[Error bounds for the SVC and SVR models]
	\label{prop:svm-error-bounds}
	
	Assume that the training set is nonempty. Then:
	
	\begin{enumerate}[label=\textnormal{(\roman*)},leftmargin=*]
		
		\item
		The models in \textnormal{(a)}, \textnormal{(c)}--\textnormal{(e)},
		and \textnormal{(g)}--\textnormal{(i)}, together with
		\eqref{eq:svm-example}, are primal instances of
		\eqref{eq:quadratic-polyhedral-primal}.
		The models in \textnormal{(b)}, \textnormal{(f)}, and
		\textnormal{(j)} admit equivalent lifted primal formulations of the
		same form.
		
		\item
		The primal and dual local error bounds and the dual Luo--Tseng error
		bound hold for all the above formulations. The primal Luo--Tseng error
		bound also holds whenever the corresponding primal objective is
		level-bounded. In particular, this is the case for all the models
		above except that no such conclusion is asserted for the
		\(L_1\)- and \(L_2\)-loss SVC models without an additional
		level-boundedness assumption.
		
	\end{enumerate}
\end{proposition}
\begin{proof}
	See Appendix~\ref{app:proof:prop:svm-error-bounds}.
\end{proof}

%
%

\end{example}

\begin{example}[Polyhedrally constrained convex quadratic programming]
	\label{ex:polyhedral-constrained-qp}
	Consider
	\begin{equation}
		\label{eq:polyhedral-constrained-qp}
		\begin{aligned}
			\min_{v\in\mathbb R^n}\quad&
			\frac12\langle v,Qv\rangle
			+
			\langle c,v\rangle
			+
			\gamma\\
			\text{\rm subject to}\quad&
			Hv+h\in C,
		\end{aligned}
	\end{equation}
	where \(Q\succeq0\), \(H\in\mathbb R^{m\times n}\), \(h\in\mathbb R^m\), and \(C\subseteq\mathbb R^m\) is a nonempty polyhedral convex set.
	
	Define \(\Omega:=\{v\in\mathbb R^n:Hv+h\in C\}.	\)
	We use the primal identity representation
	\[
	f_1(v)=\frac12\langle v,Qv\rangle+\langle c,v\rangle+\gamma,
	\qquad
	f_2(v)=\delta_\Omega(v),
	\qquad
	A=I,\quad b=0.
	\]
	
	Assume that the problem has an optimal solution. Since the feasible set is polyhedral, its indicator is a polyhedral convex function. Hence Proposition~\ref{prop:quadratic-polyhedral-error-bounds} gives the primal and dual   local error bounds and the dual Luo--Tseng error bound.
	
	If the primal objective is level-bounded, the primal Luo--Tseng error bound also holds. In particular, this is the case when \(Q\succ0\).
\end{example}

\subsection{Strongly convex smooth models with convex regularizers} \label{subsec:strongly-convex-convex-regularizers}
The previous subsection exploits polyhedrality of the regularizer and therefore allows the quadratic term to be merely convex. We now consider a complementary structural regime: the smooth term is strongly convex, while the regularizer is allowed to be an arbitrary proper closed convex function and need not be polyhedral. Thus neither of the two classes considered here contains the other. Their intersection consists, in particular, of strongly convex quadratic models with polyhedral regularizers.

Consider the composite problem
\begin{equation}
\label{eq:strongly-convex-general-regularizer}
    \min_{v\in\mathbb R^n}
    \left\{
        \Phi(v):=f(v)+P(v)
    \right\},
\tag{SC}
\end{equation}
where $f:\mathbb R^n\to\mathbb R$ is finite-valued, differentiable, and \(\sigma\)-strongly convex for some \(\sigma>0\), while $P:\mathbb R^n\to(-\infty,+\infty]$ is an arbitrary proper, closed, convex function. We do not assume that \(P\) is polyhedral.

Its Fenchel dual is
\begin{equation}
\label{eq:strongly-convex-general-dual}
    \min_{x\in\mathbb R^n}
    d_P(x)
    :=
    P^*(x)+f^*(-x).
\tag{SC-D}
\end{equation}
The associated reduced KKT mapping is
\begin{equation}
\label{eq:strongly-convex-general-kkt}
    \mathcal M_P(v,x)
    :=
    \begin{pmatrix}
        \nabla f(v)+x\\[1mm]
        v-\partial P^*(x)
    \end{pmatrix}.
\end{equation}

Since \(f\) is strongly convex and \(P\) is convex, the primal objective \(\Phi=f+P\) is strongly convex. Hence it has a unique minimizer, denoted by \(\bar v\). The associated dual solution is also unique and is given by  $\bar x=-\nabla f(\bar v).$

\begin{proposition}[Error bounds under strong convexity]
	\label{prop:strongly-convex-arbitrary-regularizer}
	
	Suppose that Assumption~\ref{ass:basic-primal-dual-setting} holds for
	the primal-dual pair
	\eqref{eq:strongly-convex-general-regularizer}--%
	\eqref{eq:strongly-convex-general-dual}.
	Assume further that \(\nabla f\) is locally Lipschitz continuous
	around the unique primal solution \(\bar v\).
	
	Then the following statements are equivalent and all hold:
	\begin{enumerate}[label=\textnormal{(\roman*)},leftmargin=*]
		
		\item The   local error bound holds at \(\bar v\) for the
		primal problem.
		
		\item The reduced KKT mapping \(\mathcal M_P\) is metrically
		subregular at \((\bar v,\bar x)\) for \(0\).
		
		\item The   local error bound holds at \(\bar x\) for the
		dual problem.
		
	\end{enumerate}
	
	Moreover:
	
	\begin{enumerate}[label=\textnormal{(\roman*)},leftmargin=*,
		resume]
		
		\item The Luo--Tseng error bound holds for the primal problem.
		
		\item The Luo--Tseng error bound holds for the dual problem.
		
	\end{enumerate}
\end{proposition}
\begin{proof}
	See Appendix~\ref{app:strongly-convex-primal-eb}.
\end{proof}

\begin{remark}[Relation with existing nonpolyhedral analyses] \label{rem:strong-convex-nonpolyhedral-literature}
Error bounds for \(\ell_{1,p}\)-regularized problems, including group Lasso, are studied in \cite{ZhouZhangSo2015,zhouso2017}. For nuclear-norm regularization, error-bound and convergence results are given in \cite{HouZhouSoLuo2013,zhouso2017}. Drusvyatskiy and Lewis~\cite{DrusvyatskiyLewis2018} develop a more general error-bound and quadratic-growth framework that also covers group-norm and nuclear-norm regularizers under appropriate regularity conditions. For convex clustering, Ho, Lin, and Jordan~\cite{HoLinJordan2022} established a global residual error bound for a class of dual filtering-clustering problems that includes \(\ell_2\)-convex clustering, and used it to establish linear convergence of first-order methods.

The error-bound results cited above are obtained under assumptions different from those imposed in Proposition~\ref{prop:strongly-convex-arbitrary-regularizer}. In this regime, the local error bound for the primal problem follows directly from strong convexity and local Lipschitz continuity of the gradient, while the   local error bound for the dual problem is obtained through metric subregularity of the reduced KKT mapping. Thus the present result does not subsume the more general non-strongly-convex cases treated in the cited works.
\end{remark}

%
%
%

\begin{example}[Weighted convex clustering	{\cite{Pelckmans2005}}]	\label{ex:weighted-convex-clustering}
	
	Let $\mathsf A=[a_1,\ldots,a_N]\in\mathbb R^{d\times N}$ be a data matrix whose columns are the observations. Consider
	\begin{equation}
		\label{eq:weighted-convex-clustering}
		\min_{X\in\mathbb R^{d\times N}}
		\left\{
		\frac12\sum_{i=1}^{N}\|x_i-a_i\|_2^2
		+
		\gamma\sum_{1\leq i<j\leq N}
		w_{ij}\|x_i-x_j\|_p
		\right\},
	\end{equation}
	where
	\[
	X=[x_1,\ldots,x_N],\qquad
	\gamma>0,\qquad
	w_{ij}=w_{ji}\geq0,\qquad
	p\geq1.
	\]
	Convex clustering based on fusion penalties dates back to Pelckmans et al.~\cite{Pelckmans2005}.
	
	Viewing \(\mathbb R^{d\times N}\) as a Euclidean space equipped with 	the Frobenius inner product, define
	\[
	f_1(X)
	:=
	\frac12\|X-\mathsf A\|_F^2,
	\]
	and
	\[
	f_2(X)
	:=
	\gamma\sum_{1\leq i<j\leq N}
	w_{ij}\|x_i-x_j\|_p.
	\]
	We use the identity composite representation $	A=I,\  b=0.$
	The function \(f_1\) is \(1\)-strongly convex and has 	\(1\)-Lipschitz continuous gradient, while \(f_2\) is finite-valued,  proper, closed, and convex. In particular, \(f_2\) need not be 	polyhedral; for example, it is generally nonpolyhedral for  \(p=2\) when \(d>1\).
	
	Therefore, Proposition~	\ref{prop:strongly-convex-arbitrary-regularizer} yields the primal 	and dual local error bounds, the two Luo--Tseng error bounds, and  metric subregularity of the reduced KKT mapping.
\end{example}

%
%
%
%

\subsection{Error bounds from KKT metric subregularity beyond
	polyhedrality and strong convexity}
\label{subsec:kkt-msr-beyond-structural}

The preceding two subsections derive error bounds from two standard
structural regimes: polyhedrality and strong convexity. We now
consider a class in which neither property is required. Instead, we
start directly from metric subregularity of the reduced KKT mapping
and use the primal--dual framework developed in
Section~\ref{sec:primal-dual-eb-relation} to derive error bounds for
both the primal and dual problems.

A representative setting is regularized least squares with a
possibly rank-deficient design matrix and a nonpolyhedral norm
regularizer. The following result isolates the KKT-to-error-bound
step needed below.
	Consider
\begin{equation}
	\label{eq:rank-deficient-regularized-ls}
	\min_{v\in\mathbb R^n}
	\left\{
	F(v)
	:=
	\frac12\|Bv-y\|^2+P(v)
	\right\},
\end{equation}
where
\[
B\in\mathbb R^{m\times n},
\qquad
y\in\mathbb R^m,
\]
and \(P:\mathbb R^n\to\mathbb R\) is a norm. No rank assumption is
imposed on \(B\).

Use the identity composite representation
\[
f_1(v)
:=
\frac12\|Bv-y\|^2,
\qquad
f_2(v):=P(v),
\qquad
A=I,
\qquad
b=0.
\]
The associated reduced KKT mapping is
\begin{equation}
	\label{eq:rank-deficient-ls-kkt}
	\mathcal M_P(v,x)
	:=
	\begin{pmatrix}
		B^\top(Bv-y)+x\\[1mm]
		v-\partial P^*(x)
	\end{pmatrix}.
\end{equation}

\begin{proposition}[KKT-based error bounds for regularized least squares]
	\label{prop:kkt-eb-rank-deficient-ls}
	Suppose that \(\mathcal M_P\) is metrically subregular at every
	\((\bar v,\bar x)\in\mathcal S\) for \(0\). Then:
	
	\begin{enumerate}[label=\textnormal{(\roman*)},leftmargin=*]
		
		\item
		The primal local error bound holds at every
		\(\bar v\in\mathcal V\).
		
		\item
		The dual local error bound holds at every
		\(\bar x\in\Lambda\).
		
		\item
		Both the primal and dual Luo--Tseng error bounds hold.
		
	\end{enumerate}
\end{proposition}
\begin{proof}
	See Appendix~\ref{app:proof:kkt-eb-rank-deficient-ls}.
\end{proof}

\begin{example}[\(\ell_{1,q}\)-regularized least squares and Group Lasso]
	\label{ex:group-lasso-kkt-msr}
	
	Let \(\mathcal J\) be a partition of \(\{1,\ldots,n\}\), and consider
	\begin{equation}
		\label{eq:l1q-regularized-ls}
		\min_{v\in\mathbb R^n}
		\left\{
		\frac12\|Bv-y\|^2
		+
		\lambda
		\sum_{J\in\mathcal J}
		\omega_J\|v_J\|_q
		\right\},
	\end{equation}
	where $	\lambda>0,\ 	\omega_J>0,\ 	q\in(1,2].$	
	No rank assumption is imposed on
	\(B\in\mathbb R^{m\times n}\). The case \(q=2\) is the standard
	Group Lasso introduced by Yuan and Lin~\cite{YuanLin2006}.
	
	Define
	\[
	P_q(v)
	:=
	\lambda
	\sum_{J\in\mathcal J}
	\omega_J\|v_J\|_q.
	\]
	Then \(P_q\) is a norm and is generally nonpolyhedral when
	\(q>1\) and at least one group has dimension greater than one.
	Moreover, if \(B\) is rank deficient, the least-squares term is not
	strongly convex. Hence this model need not fall into either of the
	two structural regimes considered in the preceding subsections.
	
	To establish metric subregularity of the reduced KKT mapping, write
	\eqref{eq:l1q-regularized-ls} equivalently as
	\[
	\begin{aligned}
		\min_{v,z}\quad&
		\frac12\|Bv-y\|^2+P_q(z)\\
		\textnormal{s.t.}\quad&
		v-z=0.
	\end{aligned}
	\]
	In the framework of Yuan, Zeng, and Zhang~
	\cite{YuanZengZhang2020}, the constraint matrix associated with
	\(v\) is \(I\), and hence their full-row-rank condition is
	automatically satisfied, independently of the rank of the design
	matrix \(B\). Moreover, the smooth term has the structured form
	\[
	h(Bv),
	\qquad
	h(u)=\frac12\|u-y\|^2,
	\]
	so their smoothness assumptions are satisfied.
	
	By \cite[Theorem~50]{YuanZengZhang2020}, the corresponding reduced
	saddle-point mapping
	\[
	T(v,\mu)
	=
	\begin{pmatrix}
		B^\top(Bv-y)-\mu\\[1mm]
		v-\partial P_q^*(-\mu)
	\end{pmatrix}
	\]
	is metrically subregular at every KKT point for
	\(q\in[1,2]\). Since
	\[
	T(v,\mu)
	=
	\mathcal M_{P_q}(v,-\mu),
	\]
	and the transformation
	\((v,\mu)\mapsto(v,-\mu)\) is an isometry,
	\(\mathcal M_{P_q}\) is metrically subregular at every corresponding
	KKT point.
	
	Therefore, Proposition~
	\ref{prop:kkt-eb-rank-deficient-ls} yields the primal and dual local
	error bounds and both Luo--Tseng error bounds for
	\eqref{eq:l1q-regularized-ls}. In particular, these conclusions hold
	for the standard Group Lasso without any full-column-rank assumption
	on the design matrix \(B\).
\end{example}

\begin{remark}[Relation with existing error-bound analyses]
	\label{rem:group-lasso-kkt-route}
	
	Primal error bounds for Group Lasso and more general
	\(\ell_{1,q}\)-regularized models have been established previously;
	see, for example,
	\cite{Tseng2010,ZhouZhangSo2015,zhouso2017}.
	The purpose of Example~\ref{ex:group-lasso-kkt-msr} is not to
	rederive these primal results through another regularizer-specific
	error-bound analysis. Instead, the starting point is the independently
	established metric subregularity of the reduced saddle-point mapping
	in \cite{YuanZengZhang2020}. Once this mapping is identified with the
	reduced KKT mapping used here, Proposition~
	\ref{prop:kkt-eb-rank-deficient-ls} yields the primal and dual error
	bounds within the same framework.
	
	Thus no separate proximal-residual error-bound argument is required
	on the primal side, and the same KKT regularity property also yields
	the corresponding dual error bounds. This applies even when the
	least-squares loss is not strongly convex and the regularizer is
	nonpolyhedral.
\end{remark}

\begin{remark}[Sparse-group Lasso]
	The same argument applies to the sparse-group Lasso regularizer
	\[
	P_{\rm sg}(v)
	:=
	\mu\|v\|_1
	+
	\lambda
	\sum_{J\in\mathcal J}
	\omega_J\|v_J\|_2,
	\]
	where $	\mu\geq0,	\ 	\lambda>0,	\ 	\omega_J>0.$
	This function is a norm, and is generally nonpolyhedral when at least
	one group has dimension greater than one. Yuan, Zeng, and Zhang~
	\cite[Theorem~50]{YuanZengZhang2020} also establish metric
	subregularity of the corresponding reduced saddle-point mapping for
	the sparse-group Lasso. Hence Proposition~
	\ref{prop:kkt-eb-rank-deficient-ls} again yields the primal and dual
	local error bounds and both Luo--Tseng error bounds, without a
	full-column-rank assumption on the design matrix.
\end{remark}

In Table~\ref{tab:existing-eb-comparison},
\(\checkmark\) indicates direct compatibility with the present
framework: the cited error-bound result can be recovered, or the
cited KKT-regularity result directly supplies the required
regularity condition; \(\triangle\) denotes partial coverage or a
related but nonidentical formulation or residual; and \(\times\)
indicates that the cited result is not covered under the current
assumptions.

\begin{sidewaystable}[p]
	\centering
	
	\fontsize{8}{11}\selectfont
	\setlength{\tabcolsep}{1.4pt}
	\renewcommand{\arraystretch}{0.82}
	
	\caption{Representative error-bound and KKT-regularity results and
		their relation to the present framework.}
	\label{tab:existing-eb-comparison}
	
	\begin{tabularx}{0.96\textheight}{
			>{\raggedright\arraybackslash}p{2.5cm}
			>{\raggedright\arraybackslash}p{2.8cm}
			>{\raggedright\arraybackslash}p{3.0cm}
			>{\raggedright\arraybackslash}X
			>{\raggedright\arraybackslash}p{4.4cm}
		}
		
		\toprule
		Reference
		& \(f_1\)
		& \(f_2\)
		& Main assumptions / result
		& Coverage by the present framework
		\\
		\midrule
		
		Luo--Tseng (1992)
		\cite{LuoTseng1992}
		\newline{\itshape Primal}
		&
		\(h(Bv)+\langle c,v\rangle\)
		&
		\(\delta_X\), \(X\) polyhedral
		&
		\(h\) strictly convex and essentially smooth, with positive
		local curvature; local projection-gradient EB.
		&
		\(\triangle\)
		Some finite-valued subclasses are covered, but not the full
		essentially smooth setting.
		\\
\cmidrule(lr){1-5}
		
		Luo--Tseng (1993)
		\cite{LuoTseng1993b}
		\newline{\itshape Dual}
		&
		\(f\)
		&
		\(\delta_X(u)+\sigma_Y(w)\), with \((u,w)=(v,q-Ev)\)
		&
		Co-finite essentially smooth \(f\), polyhedral \(X,Y\), and
		strict monotonicity and Lipschitz continuity of \(\nabla f\) on
		the relevant compact subsets; dual local EB.
		&
		\(\triangle\)
		Strongly convex finite-valued subclasses overlap with our setting,
		but the dual formulation and residual differ.
		\\
\cmidrule(lr){1-5}
		
		Tseng--Yun (2009)
		\cite{TsengYun2009}
		\newline{\itshape Primal}
		&
		Quadratic \(f\)
		&
		\(cP\), \(P\) polyhedral
		&
		One sufficient case for their local residual EB; \(f\) need not
		be convex.
		&
		\(\triangle\)
		Our quadratic--polyhedral result covers the convex subclass and
		additionally yields the associated dual EBs.
		\\
\cmidrule(lr){1-5}
		
		Tseng (2010)
		\cite{Tseng2010}
		\newline{\itshape Primal}
		&
		\(h(Bv)\)
		&
		\(\tau\sum_J\omega_J\|v_J\|_2\)
		&
		\(h\) strongly convex on its open domain, with
		\(\nabla h\) Lipschitz on compact subsets; Group-Lasso proximal EB
		without full-column-rank \(B\).
		&
		\(\triangle\)
		Finite-valued regression subclasses are covered via KKT MSR;
		least squares additionally yields the dual EBs.
		\\
\cmidrule(lr){1-5}
		
		Zhou--Zhang--So (2015)
		\cite{ZhouZhangSo2015}
		\newline{\itshape Primal}
		&
		\(h(Bv)\)
		&
		\(\sum_J\omega_J\|v_J\|_q\)
		&
		Primal EB for \(q\in[1,2]\cup\{\infty\}\); it may fail for
		\(q\in(2,\infty)\).
		&
		\(\checkmark\)
		For \(q\in[1,2]\), KKT MSR follows from
		\cite{YuanZengZhang2020}; \(q=\infty\) is PLQ.
		\\
\cmidrule(lr){1-5}
		
		Zhou--So (2017)
		\cite{zhouso2017}
		\newline{\itshape Primal}
		&
		\(h(Bv)+\langle c,v\rangle\)
		&
		General \(P\); in particular \(P(X)=\|X\|_*\)
		&
		General EB framework via bounded linear regularity and calmness.
		For the nuclear norm, an EB is obtained under a
		strict-complementarity-type condition.
		&
		\(\triangle\) General framework;
		\(\times\) noninjective nuclear-norm case.
		Several subclasses are covered, but the general-\(P\) result and the
		genuinely noninjective nuclear-norm case are not.
		\\
\cmidrule(lr){1-5}
		
		Yuan--Zeng--Zhang (2020)
		\cite{YuanZengZhang2020}
		\newline{\itshape KKT}
		&
		\(h(Lv)+\langle q,v\rangle\)
		&
		PLQ, \(\ell_{1,q}\), sparse-group Lasso
		&
		KKT MSR for PLQ regularizers and, under full-row-rank splitting,
		for \(\ell_{1,q}\), \(q\in[1,2]\), and sparse-group Lasso.
		&
		\(\checkmark\)
		Provides the KKT-MSR input used here; for \(v=z\), the splitting
		matrix is \(I\), regardless of the design-matrix rank.
		\\
\cmidrule(lr){1-5}
		
		Ho--Lin--Jordan (2022)
		\cite{HoLinJordan2022}
		\newline{\itshape Dual}
		&
		Strongly convex smooth \(f\)
		&
		\(\lambda\sum_k\|D_kv\|_p\)
		&
		Global dual EB for filtering--clustering models under their
		stated exponent/nondegeneracy conditions.
		&
		\(\triangle\)
		Our primal model is covered, but their dual formulation and
		residual differ from ours.
		\\
\cmidrule(lr){1-5}
		
		De Leone--Lazzari (2010)
		\cite{DeLeoneLazzari2010}
		\newline{\itshape Dual}
		&
		Strongly convex quadratic
		&
		Polyhedral feasible-set indicator
		&
		Global EBs for the duals of \(L_2\)-SVM and SCL2-SVM.
		&
		\(\triangle\)
		Their dual QPs are quadratic--polyhedral, but their global
		feasible-point/primal--dual error bounds are not identical to the
		local or Luo--Tseng error bounds used here.
		\\
		
		\bottomrule
		
	\end{tabularx}
	
\end{sidewaystable}

Table~\ref{tab:applications-literature} summarizes the main applications considered in this paper and compares them with representative existing optimization EB/MSR results. For lifted quadratic--polyhedral formulations, the conclusions refer to the corresponding lifted primal--dual pair. The table is intended to highlight the role of the present framework rather than to provide an exhaustive literature survey.

\begin{sidewaystable}[p]
	\centering
	
	\fontsize{8}{12}\selectfont
	\setlength{\tabcolsep}{1.2pt}
	\renewcommand{\arraystretch}{0.80}
	
	\caption{Applications considered in this paper and their relation to
		existing optimization error-bound results.}
	\label{tab:applications-literature}
	
	\begin{tabularx}{0.96\textheight}{
			>{\raggedright\arraybackslash}p{3.0cm}
			>{\raggedright\arraybackslash}p{3.0cm}
			>{\raggedright\arraybackslash}p{3.6cm}
			>{\raggedright\arraybackslash}X
			>{\raggedright\arraybackslash}p{4.5cm}
		}
		
		\toprule
		Model
		& \(f_1\)
		& \(f_2\)
		& Existing EB/MSR results
		& Present result
		\\
		\midrule
		
		Lasso-type models:
		Lasso, weighted/adaptive Lasso,
		generalized Lasso, SLOPE
		&
		\(\displaystyle
		\frac12\|y-X\beta\|^2
		\)
		&
		\(\lambda\|\beta\|_1\),
		\(\lambda\sum_jw_j|\beta_j|\),
		\(\lambda\|D\beta\|_1\), or
		\(\lambda J_w(\beta)\)
		&
		Classical quadratic--polyhedral EB results apply; see
		\cite{TsengYun2009,zhouso2017}.
		&
		Both local EBs and the dual Luo--Tseng EB hold.
		The primal Luo--Tseng EB follows under level-boundedness;
		for generalized Lasso, a sufficient condition is
		\(\ker X\cap\ker D=\{0\}\).
		\\
		\cmidrule(lr){1-5}
		
		SVC/SVR models in
		Proposition~\ref{prop:svm-error-bounds}
		&
		Quadratic smooth term,
		directly or after lifting
		&
		Piecewise-affine term or
		polyhedral indicator
		&
		Specialized EBs exist for some SVM formulations; see, e.g.,
		\cite{DeLeoneLazzari2010}.
		&
		Both local EBs and the dual Luo--Tseng EB hold for all models
		considered. The primal Luo--Tseng EB also holds for the
		level-bounded cases.
		\\
		\cmidrule(lr){1-5}
		
		Weighted convex clustering
		&
		\(\displaystyle
		\frac12\|X-\mathsf A\|_F^2
		\)
		&
		\(\displaystyle
		\gamma\sum_{i<j}w_{ij}\|x_i-x_j\|_p
		\)
		&
		Ho--Lin--Jordan~\cite{HoLinJordan2022} establish a global
		dual EB for related filtering--clustering models.
		&
		For all \(p\geq1\), both local EBs, both Luo--Tseng EBs,
		and KKT MSR follow from the strongly convex framework.
		\\
		\cmidrule(lr){1-5}
		
		\(\ell_{1,q}\)-regularized least squares,
		Group Lasso, sparse-group Lasso
		&
		\(\displaystyle
		\frac12\|Bv-y\|^2
		\)
		&
		\(\ell_{1,q}\) or sparse-group
		regularizer
		&
		Primal EBs are known from
		\cite{Tseng2010,ZhouZhangSo2015,zhouso2017};
		KKT MSR follows from
		\cite{YuanZengZhang2020}.
		&
		Even for rank-deficient \(B\) and nonpolyhedral regularizers,
		Proposition~\ref{prop:kkt-eb-rank-deficient-ls}
		gives both local EBs and both Luo--Tseng EBs.
		\\
		\cmidrule(lr){1-5}
		
		Polyhedrally constrained
		convex quadratic programming
		&
		\(\displaystyle
		\frac12\langle v,Qv\rangle+\langle c,v\rangle+\gamma
		\)
		&
		\(\delta_\Omega(v)\),
		\(\Omega\) polyhedral
		&
		Classical polyhedral and quadratic--polyhedral EB theory.
		&
		Both local EBs and the dual Luo--Tseng EB hold;
		the primal Luo--Tseng EB holds under level-boundedness,
		in particular when \(Q\succ0\).
		\\
		
		\bottomrule
		
	\end{tabularx}
	
\end{sidewaystable}

\section{Conclusion}\label{sec:conclusion}

This paper developed a primal--dual framework for error bounds in convex composite optimization based on metric subregularity of a reduced KKT mapping. Under the standing solvability and strong-duality assumptions, the zero set of this mapping is the Cartesian product of the primal and dual solution sets. This representation allows the primal proximal-gradient residual, the dual subdifferential residual, and the KKT residual to be analyzed within a common framework.

We established implications in both directions between KKT metric subregularity and the primal and dual   local error bounds. Under full row rank of the composite matrix and local Lipschitz continuity of the smooth gradient, KKT metric subregularity is equivalent to the primal  local error bound. Under strong convexity of the smooth term, it is equivalent to the dual local error bound. If the composite matrix has full row rank and the smooth term is strongly convex, KKT metric subregularity is equivalent to the simultaneous validity of the primal and dual local error bounds. We also established the relations between the local  and Luo--Tseng error bounds. The Luo--Tseng bounds imply the corresponding   local bounds, with local Lipschitz continuity of the smooth gradient required on the primal side, whereas level-boundedness allows   local bounds valid at every solution to be extended to the Luo--Tseng bounds.

The framework was applied to quadratic--polyhedral models, strongly convex smooth models with general convex regularizers, and regularized least-squares models beyond polyhedrality and strong convexity. The quadratic--polyhedral class covers Lasso-type models, SLOPE, several SVC and SVR formulations, and polyhedrally constrained convex quadratic programs. The strongly convex regime includes weighted convex clustering. For \(\ell_{1,q}\)-regularized least squares, including Group Lasso with a possibly rank-deficient design matrix, known metric subregularity of the reduced saddle-point mapping yields both primal and dual local error bounds and both Luo--Tseng error bounds through the present framework.

\section*{Statements and Declarations}

\noindent\textbf{Funding.}
The corresponding author's research is supported by National Natural Science Foundation of China (NSFC) No. 12271526.

\noindent\textbf{Competing interests.}
The authors have no relevant financial or non-financial interests to disclose.

\noindent \textbf{Data availability.}
Data sharing is not applicable to this article as no datasets were generated or analysed during the current study.

\begin{appendices}
	\section{Proofs in Section \ref{sec:Preliminaries}}
	\textbf{Proof of Proposition \ref{prop:kkt-solution-product}}\label{app:kkt-solution-product}
		Let \(v^\star\in\mathcal V\) and \(x^\star\in\Lambda\). By strong duality,
		\[
		\begin{aligned}
			0
			&=F(v^\star)+d(x^\star)\\
			&=
			\Bigl[
			f_1(v^\star)+f_1^*(-A^\top x^\star)
			+\langle A^\top x^\star,v^\star\rangle
			\Bigr]+
			\Bigl[
			f_2(Av^\star+b)+f_2^*(x^\star)
			-\langle x^\star,Av^\star+b\rangle
			\Bigr].
		\end{aligned}
		\]
		Both bracketed terms are nonnegative by  Fenchel’s inequality and therefore must vanish. By~\cite[Theorem~23.5]{Rockafellar1970},
		\[
		-A^\top x^\star\in\partial f_1(v^\star),
		\qquad
		Av^\star+b\in\partial f_2^*(x^\star).
		\]
		Since \(f_1\) is differentiable,
		\[
		\nabla f_1(v^\star)+A^\top x^\star=0.
		\]
		Thus \((v^\star,x^\star)\in\mathcal S\), and consequently $	\mathcal V\times\Lambda\subseteq\mathcal S.$
		
		Conversely, let \((v,x)\in\mathcal S\). Then
		\[
		-A^\top x=\nabla f_1(v)\in\partial f_1(v),
		\qquad
		Av+b\in\partial f_2^*(x).
		\]
		Applying Rockafellar~\cite[Theorem~23.5]{Rockafellar1970} to \(f_1\) and \(f_2^*\), and using \(f_2^{**}=f_2\), we obtain
		\[
		f_1(v)+f_1^*(-A^\top x)
		+\langle A^\top x,v\rangle=0,
		\quad
		f_2(Av+b)+f_2^*(x)
		-\langle x,Av+b\rangle=0.
		\]
		Consequently, \(F(v)+d(x)=0\). Hence
		\[
		\bigl(F(v)-F^\star\bigr)
		+
		\bigl(d(x)-d^\star\bigr)
		=0.
		\]
		Both terms are nonnegative, so \(F(v)=F^\star\) and \(d(x)=d^\star\). Therefore \(v\in\mathcal V\) and \(x\in\Lambda\), which proves$
		\mathcal S\subseteq\mathcal V\times\Lambda.$
		\hfill $\square$
		
\noindent \textbf{Proof of Proposition \ref{prop:local-sublevel-relations}}\label{app:local-sublevel-relations}
	\noindent	\textnormal{(i)}
		Define the proximal-gradient mapping $T(v):=\operatorname{prox}_{G}\bigl(v-\nabla f_1(v)\bigr),$
		so that
		\[
		 r_P(v)=T(v)-v.
		\]
		Since \(\bar v\in\mathcal V\), we have 	\( r_P(\bar v)=0\), and hence \(T(\bar v)=\bar v\).
		
		Let \(U\) be a neighborhood of \(\bar v\) on which \(\nabla f_1\) is \(L\)-Lipschitz continuous. The mapping \(T\) is continuous at \(\bar v\). Hence, after shrinking the neighborhood if necessary, both \(v\) and \(T(v)\) belong to \(U\) whenever \(v\) is sufficiently close to \(\bar v\).
		
		We first show that \(F(T(v))\) is locally bounded above. Set $w:=T(v).$ By the definition of the proximal mapping, comparison with
		\(\bar v\in\operatorname{dom}G\) gives
		\[
		\begin{aligned}
			G(w) +\frac12 \left\| w-v+\nabla f_1(v) \right\|^2
			\leq
			G(\bar v) +\frac12 \left\| \bar v-v+\nabla f_1(v) \right\|^2.
		\end{aligned}
		\]
		Consequently, \(G(w)\) is bounded above for \(v\) sufficiently close to \(\bar v\). Since \(w\to\bar v\) and \(f_1\) is continuous, there exist a finite number \(\zeta>F^\star\) and a neighborhood \(U_0\) of \(\bar v\) such that
		\begin{equation}\label{eq:prox-point-local-sublevel}
			F(T(v))\leq\zeta,
			\qquad
			\forall v\in U_0.
		\end{equation}
		
		Let \(\kappa_\zeta,\epsilon_\zeta>0\) be constants associated with the Luo–Tseng error bound on the set \(\{v:F(v)\leq\zeta\}\).
		
		Using the nonexpansiveness of the proximal mapping and the local Lipschitz continuity of \(\nabla f_1\), we obtain
		\[
		\begin{aligned}
			\| r_P(w)\|
			&=
			\|T(w)-w\|=
			\|T(w)-T(v)\|\\
			&\leq
			\left\|
			w-\nabla f_1(w)
			-
			\bigl(v-\nabla f_1(v)\bigr)
			\right\|\\
			&\leq
			(1+L)\|w-v\|=
			(1+L)\| r_P(v)\|.
		\end{aligned}
		\]
		Since \( r_P\) is continuous and \( r_P(\bar v)=0\), there exists \(\rho_P>0\) such that, whenever \(v\in\mathbb B(\bar v,\rho_P)\),
		\[
		v\in U_0,\ 
		v,T(v)\in U \ \text{and}\ 
		(1+L)\| r_P(v)\|
		\leq
		\epsilon_\zeta.
		\]
		It follows that
		\[
		F(w)\leq\zeta,
		\qquad
		\| r_P(w)\|\leq\epsilon_\zeta.
		\]
		The Luo–Tseng error bound therefore yields
		\[
		\operatorname{dist}(w,\mathcal V)
		\leq
		\kappa_\zeta\| r_P(w)\|.
		\]
		Consequently,
		\[
		\begin{aligned}
			\operatorname{dist}(v,\mathcal V)
			&\leq
			\|v-w\|
			+
			\operatorname{dist}(w,\mathcal V)\\
			&\leq
			\| r_P(v)\|
			+
			\kappa_\zeta\| r_P(w)\|\\
			&\leq
			\left[
			1+\kappa_\zeta(1+L)
			\right]
			\| r_P(v)\|.
		\end{aligned}
		\]
		Thus the   local error bound holds at \(\bar v\), with
		$	\kappa_P	:=	1+\kappa_\zeta(1+L).$

	\noindent	\textnormal{(ii)}
			Choose any \(\eta>d^\star\), and let \(\kappa_\eta,\epsilon_\eta>0\) be constants associated with the dual Luo--Tseng error bound on the sublevel set $\{x:d(x)\leq\eta\}.$ Choose \(\rho_D>0\) such that
			\begin{equation}\label{eq:dual-local-radius-choice}
				\epsilon_\eta\rho_D
				\leq
				\eta-d^\star.
			\end{equation}
			
			Take any \(x\in\mathbb B(\bar x,\rho_D)\).
			
			If \(r_D(x)=+\infty\), the desired inequality holds automatically. Suppose first that $r_D(x)>\epsilon_\eta.$
			Since \(\bar x\in\Lambda\),
			\[
			\begin{aligned}
				\operatorname{dist}(x,\Lambda)
				&\leq
				\|x-\bar x\| \leq
				\rho_D <
				\frac{\rho_D}{\epsilon_\eta}r_D(x).
			\end{aligned}
			\]
			
			It remains to consider the case $r_D(x)\leq\epsilon_\eta.$ Then \(r_D(x)<+\infty\), and hence there exists \(e\in\partial d(x)\) satisfying
			\[
			\|e\|=r_D(x).
			\]
			By the subgradient inequality and the fact that
			\(d(\bar x)=d^\star\),
			\[
			\begin{aligned}
				d(x)-d^\star
				&\leq
				\langle e,x-\bar x\rangle \leq
				\|e\|\|x-\bar x\| \leq
				\epsilon_\eta\rho_D \leq
				\eta-d^\star.
			\end{aligned}
			\]
			Therefore, $d(x)\leq\eta.$ The dual Luo--Tseng error bound applies and yields
			\[
			\operatorname{dist}(x,\Lambda)
			\leq
			\kappa_\eta r_D(x).
			\]
			
			Combining the two cases, we obtain
			\[
			\operatorname{dist}(x,\Lambda)
			\leq
			\widehat\kappa_D r_D(x),
			\qquad
			\forall x\in\mathbb B(\bar x,\rho_D),
			\]
			where
			$\widehat\kappa_D:=\max\left\{	\kappa_\eta,	\frac{\rho_D}{\epsilon_\eta}	\right\}.$
			Hence the   local error bound holds at \(\bar x\).

	\noindent	\textnormal{(iii)}
			Fix \(\zeta\geq F^\star\), and define
			$
			L_\zeta
			:=
			\{v\in\mathbb R^n:F(v)\leq\zeta\}.
			$
			Since \(F\) is closed and level-bounded, \(L_\zeta\) is compact.
			
			Suppose, to the contrary, that the Luo–Tseng error bound fails on \(L_\zeta\). Then, for every integer \(k\geq1\), there exists \(v^k\in L_\zeta\) such that
			\begin{equation}\label{eq:failure-lt-sequence}
				\| r_P(v^k)\|
				\leq
				\frac1k,
				\qquad
				\operatorname{dist}(v^k,\mathcal V)
				>
				k\| r_P(v^k)\|.
			\end{equation}
			
			By compactness of \(L_\zeta\), after passing to a subsequence, there exists \(\widehat v\in L_\zeta\) such that $v^k\to\widehat v.$ The mapping \( r_P\) is continuous because \(\operatorname{prox}_G\) is nonexpansive and \(\nabla f_1\) is continuous. Since \( r_P(\widehat v)=0\), we have \(\widehat v\in\mathcal V\).
			
			By assumption, the   local error bound holds at \(\widehat v\). Thus there exist constants \(\widehat\rho_P,\widehat\kappa_P>0\) such that
			\[
			\operatorname{dist}(v,\mathcal V)
			\leq
			\widehat\kappa_P\| r_P(v)\|,
			\qquad
			\forall v\in
			\mathbb B(\widehat v,\widehat\rho_P).
			\]
			Since \(v^k\to\widehat v\), this inequality applies to \(v^k\) for all sufficiently large \(k\). It follows that
			\[
			\operatorname{dist}(v^k,\mathcal V)
			\leq
			\widehat\kappa_P\| r_P(v^k)\|,
			\]
			which contradicts \eqref{eq:failure-lt-sequence} whenever \(k>\widehat\kappa_P\).
			
			Therefore, the Luo–Tseng error bound holds on \(L_\zeta\). Since \(\zeta\geq F^\star\) was arbitrary, the conclusion follows.

	\noindent	\textnormal{(iv)}
			Fix \(\eta\geq d^\star\), and define $L_\eta^D:=\{x\in\mathbb R^m:d(x)\leq\eta\}.$ Since \(d\) is closed and level-bounded, \(L_\eta^D\) is compact.
			
			Suppose, to the contrary, that the dual Luo--Tseng error bound fails on \(L_\eta^D\). Then, for every integer \(k\geq1\), there exists
			\(x^k\in L_\eta^D\) such that
			\begin{equation}\label{eq:failure-dual-eb-sequence}
				r_D(x^k)
				\leq
				\frac1k,
				\qquad
				\operatorname{dist}(x^k,\Lambda)
				>
				k r_D(x^k).
			\end{equation}
			
			By compactness, after passing to a subsequence, there exists \(\widehat x\in L_\eta^D\) such that $x^k\to\widehat x.$
			Since \(r_D(x^k)<+\infty\), for every \(k\) we may choose \(e^k\in\partial d(x^k)\) satisfying
			\[
			\|e^k\|=r_D(x^k).
			\]
			It follows that $e^k\to0.$ The graph of the subdifferential of a proper closed convex function is closed. Therefore, $0\in\partial d(\widehat x),$ and hence $\widehat x\in\Lambda.$
			
			By assumption, the   local error bound holds at \(\widehat x\). Thus there exist constants \(\widehat\rho_D,\widehat\kappa_D>0\) such that
			\[
			\operatorname{dist}(x,\Lambda)
			\leq
			\widehat\kappa_D r_D(x),
			\qquad
			\forall x\in
			\mathbb B(\widehat x,\widehat\rho_D).
			\]
			Since \(x^k\to\widehat x\), this inequality applies to \(x^k\) for all sufficiently large \(k\). Therefore,
			\[
			\operatorname{dist}(x^k,\Lambda)
			\leq
			\widehat\kappa_D r_D(x^k),
			\]
			contradicting \eqref{eq:failure-dual-eb-sequence} whenever \(k>\widehat\kappa_D\).
			
			Thus the dual Luo--Tseng error bound holds on \(L_\eta^D\). Since \(\eta\geq d^\star\) was arbitrary, the conclusion follows.\hfill $\square$
			
			\noindent \textbf{Proof of Corollary \ref{cor:level-boundedness-sufficient-conditions}}\label{app:cor:level-boundedness-sufficient-conditions}
			(i) Suppose first that \(f_1\) is \(\mu\)-strongly convex for some \(\mu>0\). Since $G(v)$ is convex, the primal objective \(F=f_1+G\) is also \(\mu\)-strongly convex. Let \(\bar v\in\mathcal V\). Since \(0\in\partial F(\bar v)\), strong convexity gives
			\[
			F(v)
			\geq
			F(\bar v)+\frac{\mu}{2}\|v-\bar v\|^2,
			\qquad
			\forall v\in\mathbb R^n.
			\]
			Hence every lower level set of \(F\) is bounded, and therefore \(F\) is level-bounded.
			
			(ii) Suppose next that \(A\) has full row rank. Fix any \(\bar v\in\mathcal V\), and let \(x_1,x_2\in\Lambda\).
			By Proposition~\ref{prop:kkt-solution-product}, $(\bar v,x_1),(\bar v,x_2)\in\mathcal S.$
			The stationarity condition for the reduced KKT system therefore gives
			\[
			\nabla f_1(\bar v)+A^\top x_1=0,
			\qquad
			\nabla f_1(\bar v)+A^\top x_2=0.
			\]
			Consequently,
			\[
			A^\top(x_1-x_2)=0.
			\]
			Since \(A\) has full row rank, \(A^\top\) is injective, and hence \(x_1=x_2\). 
			
			Thus \(\Lambda=\{\bar x\}\) is a singleton. Suppose, to the contrary, that \(d\) is not level-bounded. Then there exist \(\alpha\in\mathbb R\) and a sequence \(\{x^k\}\) such that
			\[
			d(x^k)\leq\alpha, \qquad \|x^k-\bar x\|\to\infty.
			\]
			Set $t_k:=\|x^k-\bar x\|, \ u^k:=\frac{x^k-\bar x}{t_k}$.
			Passing to a subsequence if necessary, we may assume that \(u^k\to u\) with \(\|u\|=1\). For any fixed \(t>0\), convexity gives, for all sufficiently large \(k\),
			\[
			\begin{aligned}
				d(\bar x+t u^k)
				&\leq
				\left(1-\frac{t}{t_k}\right)d(\bar x) +
				\frac{t}{t_k}d(x^k).
			\end{aligned}
			\]
			Letting \(k\to\infty\) and using the closedness of \(d\), we obtain
			\[
			d(\bar x+t u)\leq d(\bar x)=d^\star.
			\]
			Hence \(\bar x+t u\in\Lambda\) for every \(t>0\), contradicting \(\Lambda=\{\bar x\}\). Therefore \(d\) is level-bounded.

		\section{Proofs in Section \ref{sec:primal-dual-eb-relation}}
		\textbf{Proof of Lemma \ref{lem:primal-residual-controls-kkt}}\label{app:proof:lem:primal-residual-controls-kkt}
			Let $w:=\operatorname{prox}_{G}(v-\nabla f_1(v)).$
			Then \( r_P(v)=w-v\). By the optimality condition of the proximal subproblem,
			\[
			0\in w-(v-\nabla f_1(v))+\partial G(w).
			\]
			Equivalently,
			\[
			-\nabla f_1(v)- r_P(v)\in\partial G(w).
			\]
			Since \(G=f_2\circ(A\cdot+b)\), the exact chain rule gives
			\[
			\partial G(w)=A^\top\partial f_2(Aw+b).
			\]
			Therefore, there exists \(x_P(v)\in\partial f_2(Aw+b)\) such that
			\[
			A^\top x_P(v)=-\nabla f_1(v)- r_P(v).
			\]
			Hence
			\[
			\nabla f_1(v)+A^\top x_P(v)=- r_P(v).
			\]
			Moreover, since \(x_P(v)\in\partial f_2(Aw+b)\), the conjugacy relation gives
			\[
			Aw+b\in\partial f_2^*(x_P(v)).
			\]
			Thus
			\[
			\begin{aligned}
				r_{\rm KKT}(v,x_P(v))^2
				&=
				\|\nabla f_1(v)+A^\top x_P(v)\|^2
				+
				\operatorname{dist}^2(Av+b,\partial f_2^*(x_P(v))) \\
				&\le
				\| r_P(v)\|^2
				+
				\|Av+b-(Aw+b)\|^2 \\
				&=
				\| r_P(v)\|^2+\|A(v-w)\|^2 \\
				&\le
				(1+\|A\|^2)\| r_P(v)\|^2 .
			\end{aligned}
			\]
			Taking square roots yields
			\[
			r_{\rm KKT}(v,x_P(v))
			\le
			\sqrt{1+\|A\|^2}\,\| r_P(v)\|.
			\]\hfill $\square$
			
			\noindent \textbf{Proof of Corollary \ref{cor:full-row-rank-kkt-msr-implies-primal-eb} }\label{app:proof:cor:full-row-rank-kkt-msr-implies-primal-eb}
	Let \(\kappa_{\rm KKT}>0\) and \(\epsilon_{\rm KKT}>0\) be constants associated with the metric subregularity of \(\mathcal M\) at \((\bar v,\bar x)\) for \(0\).
	
	Since \(A\) has full row rank, the affine mapping $v\mapsto Av+b$ is surjective. Hence the exact subdifferential chain rule
	\[
	\partial(f_2\circ(A\cdot+b))(w)
	=
	A^\top\partial f_2(Aw+b)
	\]
	holds for every \(w\in\mathbb R^n\).
	
	For every \(v\in\mathbb R^n\), choose \(x_P(v)\) as in Lemma~\ref{lem:primal-residual-controls-kkt}. Then
	\[
	A^\top x_P(v) = -\nabla f_1(v)- r_P(v).
	\]
	Since \(A^\top\) is injective, \(x_P(v)\) is uniquely determined and satisfies
	\begin{equation}\label{eq:explicit-primal-selection}
		x_P(v)
		=
		-(AA^\top)^{-1}A
		\bigl(\nabla f_1(v)+ r_P(v)\bigr).
	\end{equation}
	
	The mappings \(\nabla f_1\) and \( r_P\) are continuous, and hence \(x_P(\cdot)\) is continuous. Moreover, $r_P(\bar v)=0$
	because \(\bar v\in\mathcal V\), while
	\[
	\nabla f_1(\bar v)+A^\top\bar x=0
	\]
	because \((\bar v,\bar x)\in\mathcal S\). Therefore,
	\[
	\begin{aligned}
		x_P(\bar v)
		&=
		-(AA^\top)^{-1}A\nabla f_1(\bar v) =
		(AA^\top)^{-1}AA^\top\bar x =
		\bar x.
	\end{aligned}
	\]
	
	By the continuity of \(x_P(\cdot)\) at \(\bar v\), there exists
	\(\rho_P>0\) such that
	\[
	(v,x_P(v))
	\in
	\mathbb B((\bar v,\bar x),\epsilon_{\rm KKT})
	\]
	for every \(v\in\mathbb B(\bar v,\rho_P)\). Thus the localization condition in Theorem~\ref{thm:kkt-msr-implies-primal-eb} is satisfied. The conclusion follows from that theorem.\hfill $\square$
	
	\noindent \textbf{Proof of Lemma \ref{lem:dual-residual-controls-kkt}}\label{app:proof:lem:dual-residual-controls-kkt}
	Take any \(x\) with \(r_D(x)<+\infty\). Since
	\[
	r_D(x)=\operatorname{dist}(0,\partial d(x)),
	\]
	and \(\partial d(x)\) is closed and convex whenever it is nonempty, there exists \(e_D(x)\in\partial d(x)\) such that
	\[
	\|e_D(x)\|=r_D(x).
	\]
	By the assumed exact subdifferential formula, there exist
	\[
	u_D(x)\in\partial f_2^*(x),
	\qquad
	v_D(x)\in\partial f_1^*(-A^\top x)
	\]
	such that
	\[
	e_D(x)=u_D(x)-Av_D(x)-b.
	\]
	Since \(v_D(x)\in\partial f_1^*(-A^\top x)\), the conjugacy relation gives
	\[
	-A^\top x\in\partial f_1(v_D(x)).
	\]
	Because \(f_1\) is differentiable, we have
	\[
	\nabla f_1(v_D(x))+A^\top x=0.
	\]
	Moreover, since \(u_D(x)\in\partial f_2^*(x)\),
	\[
	\operatorname{dist}(Av_D(x)+b,\partial f_2^*(x))
	\le
	\|Av_D(x)+b-u_D(x)\|
	=
	\|e_D(x)\|.
	\]
	Therefore,
	\[
	\begin{aligned}
		r_{\rm KKT}(v_D(x),x)^2
		&=
		\|\nabla f_1(v_D(x))+A^\top x\|^2
		+
		\operatorname{dist}^2(Av_D(x)+b,\partial f_2^*(x)) \\
		&\le
		\|e_D(x)\|^2
		=
		r_D(x)^2 .
	\end{aligned}
	\]
	Taking square roots yields
	\[
	r_{\rm KKT}(v_D(x),x)
	\le
	r_D(x).
	\]\hfill $\square$
	
	\noindent \textbf{Proof of Corollary \ref{cor:strong-convexity-kkt-msr-implies-dual-eb} }\label{app:proof:cor:strong-convexity-kkt-msr-implies-dual-eb}
	Let \(\kappa_{\rm KKT}>0\) and \(\epsilon_{\rm KKT}>0\) be constants associated with the metric subregularity of \(\mathcal M\) at \((\bar v,\bar x)\) for \(0\).
	
	Since \(f_1\) is \(\sigma\)-strongly convex, \(f_1^*\) is continuously differentiable on \(\mathbb R^n\), and \(\nabla f_1^*\) is \(1/\sigma\)-Lipschitz continuous. Moreover,
	\[
	\partial d(x)
	=
	\partial f_2^*(x)
	-
	A\nabla f_1^*(-A^\top x)
	-
	b.
	\]
	Hence the exact subdifferential formula in Lemma~\ref{lem:dual-residual-controls-kkt} holds. The corresponding primal vector is uniquely determined by
	\begin{equation}\label{eq:explicit-dual-selection}
		v_D(x)
		:=
		\nabla f_1^*(-A^\top x).
	\end{equation}
	
	Since \((\bar v,\bar x)\in\mathcal S\),
	\[
	-A^\top\bar x
	=
	\nabla f_1(\bar v).
	\]
	By conjugacy,
	\[
	v_D(\bar x)
	=
	\nabla f_1^*(-A^\top\bar x)
	=
	\bar v.
	\]
	Furthermore,
	\[
	\begin{aligned}
		\|v_D(x)-\bar v\|
		&=
		\left\|
		\nabla f_1^*(-A^\top x)
		-
		\nabla f_1^*(-A^\top\bar x)
		\right\|\\
		&\leq
		\frac{\|A\|}{\sigma}\|x-\bar x\|.
	\end{aligned}
	\]
	Consequently,
	\[
	\begin{aligned}
		\|(v_D(x),x)-(\bar v,\bar x)\|^2
		&=
		\|v_D(x)-\bar v\|^2+\|x-\bar x\|^2\\
		&\leq
		\left(
		1+\frac{\|A\|^2}{\sigma^2}
		\right)
		\|x-\bar x\|^2.
	\end{aligned}
	\]
	
	Choose \(\rho_D>0\) such that
	\[
	\rho_D
	\sqrt{1+\frac{\|A\|^2}{\sigma^2}}
	<
	\epsilon_{\rm KKT}.
	\]
	Then $(v_D(x),x)\in\mathbb B((\bar v,\bar x),\epsilon_{\rm KKT})$ for every \(x\in\mathbb B(\bar x,\rho_D)\). Thus the localization condition in Theorem~\ref{thm:kkt-msr-implies-dual-eb} is satisfied. The conclusion follows from that theorem.\hfill $\square$
	
	\noindent \textbf{Proof of Theorem	 \ref{thm:primal-dual-eb-implies-kkt-msr}}\label{app:proof:thm:primal-dual-eb-implies-kkt-msr}
	Set $\rho_0:=\min\{\rho_P,\rho_D,\rho_R\},$ and take any $(v,x)\in\mathbb B((\bar v,\bar x),\rho_0).$
	
	If \(r_{\rm KKT}(v,x)=+\infty\), then
	\[
	\operatorname{dist}((v,x),\mathcal S)
	\leq
	\kappa_{\rm KKT}r_{\rm KKT}(v,x)
	\]
	holds automatically.
	
	Suppose now that \(r_{\rm KKT}(v,x)<+\infty\). Since
	\[
	\|v-\bar v\|
	\leq
	\|(v,x)-(\bar v,\bar x)\|
	\leq \rho_0 \leq \rho_P,
	\]
	and similarly
	\[
	\|x-\bar x\|\leq\rho_0 \leq \rho_D,
	\]
	the   local error bounds hold at \(\bar v\) and \(\bar x\) for the primal and dual problems.
	
	Moreover, since \(\rho_0\leq\rho_R\), the reverse residual comparisons \eqref{eq:reverse-residual-comparisons} also apply. Therefore, using \eqref{eq:kkt-product-distance}, we obtain
	\[
	\begin{aligned}
		\operatorname{dist}^2((v,x),\mathcal S)
		&=
		\operatorname{dist}^2(v,\mathcal V)
		+
		\operatorname{dist}^2(x,\Lambda)\\
		&\leq
		\kappa_P^2\| r_P(v)\|^2
		+
		\kappa_D^2r_D(x)^2\\
		&\leq
		\left(
		\kappa_P^2c_P^2
		+
		\kappa_D^2c_D^2
		\right)
		r_{\rm KKT}(v,x)^2.
	\end{aligned}
	\]
	Taking square roots gives
	\[
	\operatorname{dist}((v,x),\mathcal S)
	\leq
	\sqrt{\kappa_P^2c_P^2+\kappa_D^2c_D^2}\,
	r_{\rm KKT}(v,x).
	\]
	
	Thus the reduced KKT mapping \(\mathcal M\) is metrically subregular at \((\bar v,\bar x)\) for \(0\), with $\kappa_{\rm KKT}=\sqrt{\kappa_P^2c_P^2+\kappa_D^2c_D^2}.$\hfill $\square$
	
	\noindent \textbf{Proof of Proposition \ref{prop:structural-residual-comparisons}}\label{app:lem:full-row-rank-primal-residual-comparison}
  			\textnormal{(i)}
			Fix any \((v,x)\in\mathbb R^n\times\mathbb R^m\), and set $a:=\|\nabla f_1(v)+A^\top x\|, \ c:=\operatorname{dist}(Av+b,\partial f_2^*(x)).$
			Then
			\[
			r_{\rm KKT}(v,x)^2=a^2+c^2.
			\]
			If \(r_{\rm KKT}(v,x)=+\infty\), the desired inequality holds trivially. Hence, suppose that \(r_{\rm KKT}(v,x)<+\infty\). Since \(\partial f_2^*(x)\) is nonempty, closed, and convex, there exists \(u\in\partial f_2^*(x)\) such that
			\[
			\|Av+b-u\|=c.
			\]
			
			Define $\widehat v:=v+A^\dagger(u-Av-b).$ Since \(AA^\dagger=I_m\), we have
			\[
			A\widehat v+b=u,
			\qquad
			\|\widehat v-v\|
			\leq
			\|A^\dagger\|c.
			\]
			Since \(u\in\partial f_2^*(x)\), the conjugacy relation gives \(x\in\partial f_2(u)\). Therefore, for every \(z\in\mathbb R^n\),
			\[
			\begin{aligned}
				G(z)
				&=
				f_2(Az+b) \geq
				f_2(u)+\langle x,Az+b-u\rangle =
				G(\widehat v)+\langle A^\top x,z-\widehat v\rangle.
			\end{aligned}
			\]
			Thus \(A^\top x\in\partial G(\widehat v)\), and consequently
			\[
			\widehat v
			=
			\operatorname{prox}_G(\widehat v+A^\top x).
			\]
			
			Using the nonexpansiveness of \(\operatorname{prox}_G\), we obtain
			\[
			\begin{aligned}
				\| r_P(v)\|
				&=
				\left\|
				\operatorname{prox}_G(v-\nabla f_1(v))-v
				\right\|\\
				&\leq
				\left\|
				\operatorname{prox}_G(v-\nabla f_1(v))
				-
				\operatorname{prox}_G(\widehat v+A^\top x)
				\right\|
				+
				\|\widehat v-v\|\\
				&\leq
				\|v-\nabla f_1(v)-\widehat v-A^\top x\|
				+
				\|\widehat v-v\|\\
				&\leq
				a+2\|\widehat v-v\|\\
				&\leq
				a+2\|A^\dagger\|c\\
				&\leq
				\sqrt{1+4\|A^\dagger\|^2}\,
				\sqrt{a^2+c^2}.
			\end{aligned}
			\]


	\textnormal{(ii)}
	Fix any \((v,x)\in\mathbb R^n\times\mathbb R^m\), and set $a:=\|\nabla f_1(v)+A^\top x\|, \  c:=\operatorname{dist}(Av+b,\partial f_2^*(x)).$
	Then $r_{\rm KKT}(v,x)^2=a^2+c^2.$
	If \(r_{\rm KKT}(v,x)=+\infty\), the desired inequality holds trivially. Hence, suppose that \(r_{\rm KKT}(v,x)<+\infty\). Since \(\partial f_2^*(x)\) is nonempty, closed, and convex, there exists \(u\in\partial f_2^*(x)\) such that
	\[
	\|Av+b-u\|=c.
	\]
	
	Since \(f_1\) is \(\sigma\)-strongly convex, \(f_1^*\) is finite and differentiable on \(\mathbb R^n\), and \(\nabla f_1^*\) is \(1/\sigma\)-Lipschitz continuous. Let $v_x:=\nabla f_1^*(-A^\top x).$ By conjugacy,
	\[
	\nabla f_1(v_x)+A^\top x=0.
	\]
	The strong monotonicity of \(\nabla f_1\) gives
	\[
	\begin{aligned}
		\sigma\|v-v_x\|^2
		&\leq
		\left\langle
		\nabla f_1(v)-\nabla f_1(v_x),v-v_x
		\right\rangle\\
		&=
		\left\langle
		\nabla f_1(v)+A^\top x,v-v_x
		\right\rangle\\
		&\leq
		a\|v-v_x\|.
	\end{aligned}
	\]
	If \(v=v_x\), then the estimate
	\[
	\|v-v_x\|
	\leq
	\frac{a}{\sigma}
	\]
	holds trivially. Otherwise, dividing the preceding inequality by \(\|v-v_x\|\) gives the same estimate.
	Moreover, since \(x\mapsto f_1^*(-A^\top x)-\langle b,x\rangle\) is finite and continuously differentiable,
	\[
	\partial d(x)
	=
	\partial f_2^*(x)-Av_x-b.
	\]
	Thus \(u-Av_x-b\in\partial d(x)\), and consequently
	\[
	\begin{aligned}
		r_D(x)
		&\leq
		\|u-Av_x-b\|\\
		&\leq
		\|u-Av-b\|+\|A\|\|v-v_x\|\\
		&\leq
		c+\frac{\|A\|}{\sigma}a\\
		&\leq
		\sqrt{1+\frac{\|A\|^2}{\sigma^2}}\,
		\sqrt{a^2+c^2}.
	\end{aligned}
	\]
	The desired inequality follows.\hfill $\square$
	
	\noindent \textbf{Proof of Theorem \ref{thm:primal-eb-implies-kkt-msr}} \label{app:proof-primal-eb-implies-kkt-msr}

	Set $\rho:=\min\{\rho_P,\rho_1\},$ and take any $(v,x)\in\mathbb B((\bar v,\bar x),\rho).$
	
	If \(r_{\rm KKT}(v,x)=+\infty\), then the desired metric-subregularity inequality holds automatically. Suppose, therefore, that \(r_{\rm KKT}(v,x)<+\infty\).
	
	Since
	\[
	\|v-\bar v\|
	\leq
	\|(v,x)-(\bar v,\bar x)\|
	\leq \rho \leq \rho_P,
	\]
	the   local error bound applies. Combining it with \eqref{eq:primal-residual-kkt-domination}, we obtain
	\[
	\begin{aligned}
		\operatorname{dist}(v,\mathcal V)
		&\leq
		\kappa_P\| r_P(v)\|\\
		&\leq
		\kappa_Pc_Pr_{\rm KKT}(v,x).
	\end{aligned}
	\]
	
	Using \eqref{eq:dual-distance-primal-control}, we further obtain
	\[
	\begin{aligned}
		\operatorname{dist}(x,\Lambda)
		&\leq
		\alpha r_{\rm KKT}(v,x)
		+
		\beta\operatorname{dist}(v,\mathcal V)\\
		&\leq
		\left(
		\alpha+\beta\kappa_Pc_P
		\right)
		r_{\rm KKT}(v,x).
	\end{aligned}
	\]
	
	By Proposition~\ref{prop:kkt-solution-product}, $\mathcal S=\mathcal V\times\Lambda.$ Hence, under the Euclidean product norm,
	\[
	\begin{aligned}
		\operatorname{dist}^2((v,x),\mathcal S)
		&=
		\operatorname{dist}^2(v,\mathcal V)
		+
		\operatorname{dist}^2(x,\Lambda)\\
		&\leq
		\left[
		\kappa_P^2c_P^2
		+
		\left(
		\alpha+\beta\kappa_Pc_P
		\right)^2
		\right]
		r_{\rm KKT}(v,x)^2.
	\end{aligned}
	\]
	Taking square roots gives
	\[
	\operatorname{dist}((v,x),\mathcal S)
	\leq
	\sqrt{
		\kappa_P^2c_P^2
		+
		\left(
		\alpha+\beta\kappa_Pc_P
		\right)^2
	}\,
	r_{\rm KKT}(v,x).
	\]
	
	Therefore, \(\mathcal M\) is metrically subregular at
	\((\bar v,\bar x)\) for \(0\).\hfill $\square$
	
	\noindent \textbf{Proof of Theorem \ref{thm:dual-eb-implies-kkt-msr}}\label{app:proof:thm:dual-eb-implies-kkt-msr}
	Set $\rho:=\min\{\rho_D,\rho_1\},$ and take any $(v,x)\in\mathbb B((\bar v,\bar x),\rho).$
	
	If \(r_{\rm KKT}(v,x)=+\infty\), then the desired metric-subregularity inequality holds automatically. Suppose, therefore, that \(r_{\rm KKT}(v,x)<+\infty\).
	
	Since
	\[
	\|x-\bar x\|
	\leq
	\|(v,x)-(\bar v,\bar x)\|
	\leq \rho \leq \rho_D,
	\]
	the   local error bound applies. Moreover, \eqref{eq:dual-residual-kkt-domination} gives
	\[
	r_D(x)
	\leq
	c_Dr_{\rm KKT}(v,x)
	<
	+\infty.
	\]
	Consequently,
	\[
	\begin{aligned}
		\operatorname{dist}(x,\Lambda)
		&\leq
		\kappa_Dr_D(x)\leq
		\kappa_Dc_Dr_{\rm KKT}(v,x).
	\end{aligned}
	\]
	
	Using \eqref{eq:primal-distance-dual-control}, we further obtain
	\[
	\begin{aligned}
		\operatorname{dist}(v,\mathcal V)
		&\leq
		\alpha r_{\rm KKT}(v,x)
		+
		\beta\operatorname{dist}(x,\Lambda)\leq
		\left(
		\alpha+\beta\kappa_Dc_D
		\right)
		r_{\rm KKT}(v,x).
	\end{aligned}
	\]
	
	By Proposition~\ref{prop:kkt-solution-product}, $\mathcal S=\mathcal V\times\Lambda.$ Hence
	\[
	\begin{aligned}
		\operatorname{dist}^2((v,x),\mathcal S)
		&=
		\operatorname{dist}^2(v,\mathcal V)
		+
		\operatorname{dist}^2(x,\Lambda)\\
		&\leq
		\left[
		\left(
		\alpha+\beta\kappa_Dc_D
		\right)^2
		+
		\kappa_D^2c_D^2
		\right]
		r_{\rm KKT}(v,x)^2.
	\end{aligned}
	\]
	Taking square roots gives
	\[
	\operatorname{dist}((v,x),\mathcal S)
	\leq
	\sqrt{
		\left(
		\alpha+\beta\kappa_Dc_D
		\right)^2
		+
		\kappa_D^2c_D^2
	}\,
	r_{\rm KKT}(v,x).
	\]
	
	Therefore, \(\mathcal M\) is metrically subregular at \((\bar v,\bar x)\) for \(0\). \hfill $\square$

	\section{Proofs in Section \ref{sec:applications}}
	\noindent \textbf{Proof of Proposition \ref{prop:quadratic-polyhedral-error-bounds}} \label{app:prop:quadratic-polyhedral-error-bounds}
	\noindent 	\textnormal{(i)}
		By Lemma~\ref{lem:quadratic-polyhedral-kkt},	\(\mathcal M_Q\) is metrically subregular at every KKT solution.	Since the corresponding composite matrix is \(A=I\),	Corollary~\ref{cor:full-row-rank-kkt-msr-implies-primal-eb}	yields the primal   local error bound.
		
	\noindent 	\textnormal{(ii)}
		The polyhedrality of \(\psi\) implies that \(\psi^*\) is	polyhedral, while \(q^*(-\cdot)\) is piecewise linear-quadratic.	Hence \(d_Q\) is a proper closed convex piecewise linear-quadratic	function, so \(\partial d_Q\) is piecewise polyhedral; see	\cite[Proposition~12.30]{RockafellarWets1998}.	Robinson's result therefore yields the dual   local	error bound.
		
	\noindent 	\textnormal{(iii)}
		By \textnormal{(i)} and	Proposition~\ref{prop:local-sublevel-relations} \textnormal{(iii)},	level-boundedness of \(F_Q\) yields the primal Luo--Tseng	error bound.
		
	\noindent 	\textnormal{(iv)}
		Since \(A=I\) has full row rank,	Corollary~\ref{cor:level-boundedness-sufficient-conditions}	implies that \(d_Q\) is level-bounded.	Together with \textnormal{(ii)},	Proposition~\ref{prop:local-sublevel-relations} \textnormal{(iv)}	gives the dual Luo--Tseng error bound.
		
	\noindent 	\textnormal{(v)}
		If \(Q\succ0\), then \(F_Q\) is strongly convex and hence	level-bounded. The conclusion follows from	\textnormal{(i)}--\textnormal{(iv)}.	\hfill $\square$

\noindent
\textbf{Proof of Proposition \ref{prop:svm-error-bounds}}
\label{app:proof:prop:svm-error-bounds}.

All the models considered here are analyzed from the primal side.

Let $v:=(\omega,\beta)\in\mathbb R^{d+1}.$ For the \(L_1\)-loss, pinball-loss, and \(\epsilon\)-insensitive pinball-loss SVC models, set
\[
f_1(v):=\frac12\|\omega\|^2,
\]
and
\[
f_2(v)
:=
C\sum_{i=1}^m
\ell\!\left(
1-y_i(\omega^\top z_i+\beta)
\right),
\]
where \(\ell\) denotes the corresponding scalar loss. We use $A=I,\ b=0.$
Since \(f_1\) is a finite-valued convex quadratic function and each
of these losses is finite-valued, convex, and piecewise affine,
\(f_2\) is a proper, closed, polyhedral convex function.

For the \(\epsilon\)-insensitive-loss and
\(\epsilon\)-penalty-loss SVR models, the same representation is used,
with
\[
f_2(v)
:=
C\sum_{i=1}^m
L\!\left(
y_i-\omega^\top z_i-\beta
\right).
\]
Again, \(f_2\) is proper, closed, and polyhedral convex.

For the least-squares SVC, take
\[
f_1(v)
:=
\frac12\|\omega\|^2
+
C\sum_{i=1}^m
\left(
1-y_i(\omega^\top z_i+\beta)
\right)^2,
\]
together with $f_2\equiv0, \  A=I, \ b=0.$
The function \(f_1\) is a finite-valued convex quadratic function.
Moreover, its quadratic form in a direction
\((d_\omega,d_\beta)\) is
\[
\|d_\omega\|^2
+
2C\sum_{i=1}^m
\bigl(
z_i^\top d_\omega+d_\beta
\bigr)^2,
\]
where we have used \(y_i^2=1\). If this quantity is zero, then
\(d_\omega=0\), and, since the training set is nonempty,
\(d_\beta=0\). Hence \(f_1\) is strongly convex.

Similarly, for the least-squares SVR, take
\[
f_1(v)
:=
\frac12\|\omega\|^2
+
C\sum_{i=1}^m
\left(
y_i-\omega^\top z_i-\beta
\right)^2,
\]
with $
f_2\equiv0,
\ 
A=I,
\ 
b=0.$
The same argument shows that \(f_1\) is strongly convex.

Finally, for the bias-regularized hinge-loss SVC
\eqref{eq:svm-example}, take
\[
f_1(v)
:=
\frac12\|\omega\|^2+\frac12\beta^2,
\]
and
\[
f_2(v)
:=
C\sum_{i=1}^m
\max\left\{
1-y_i(\omega^\top z_i+\beta),0
\right\},
\]
with $A=I,\ b=0.$
Here \(f_1\) is strongly convex quadratic and \(f_2\) is
finite-valued polyhedral convex. Thus all the models listed as direct
formulations in Proposition~\ref{prop:svm-error-bounds} are instances
of \eqref{eq:quadratic-polyhedral-primal}.

\medskip
\noindent
\emph{Lifted formulations.}
For the \(L_2\)-loss SVC, we use
\[
(t)_+^2
=
\min_{\xi}
\left\{
\xi^2:
\xi\geq t,\;
\xi\geq0
\right\}.
\]
Hence the model is equivalent to
\[
\begin{aligned}
	\min_{\omega,\beta,\xi}\quad&
	\frac12\|\omega\|^2
	+
	C\sum_{i=1}^m\xi_i^2\\
	\textnormal{subject to}\quad&
	\xi_i
	\geq
	1-y_i(\omega^\top z_i+\beta),
	\qquad i=1,\ldots,m,\\
	&
	\xi_i\geq0,
	\qquad i=1,\ldots,m.
\end{aligned}
\]
Let
\[
\widetilde v:=(\omega,\beta,\xi)
\]
and let \(\Omega_{L_2}\) denote the feasible set above. Then
\[
f_1(\widetilde v)
:=
\frac12\|\omega\|^2
+
C\sum_{i=1}^m\xi_i^2,
\qquad
f_2(\widetilde v)
:=
\delta_{\Omega_{L_2}}(\widetilde v),
\]
with $A=I, \  b=0.$
Thus the lifted \(L_2\)-loss SVC is of the form
\eqref{eq:quadratic-polyhedral-primal}.

For the Huber-pinball SVC, recall that
\[
\ell_{\rm hp}^{\tau,\delta}(t)
=
h_\delta(t)+\tau h_\delta(-t),
\]
where
\[
h_\delta(t)
=
\min_{u,s}
\left\{
\frac{u^2}{2\delta}+s:
u+s\geq t,\;
u\geq0,\;
s\geq0
\right\}.
\]
The two minimization problems defining
\(h_\delta(t)\) and \(h_\delta(-t)\) involve disjoint auxiliary
variables. Hence their optimal values can be combined into a joint
minimization. Moreover, for each fixed \((\omega,\beta)\), the
resulting auxiliary minimization attains its minimum. Therefore,
minimizing its optimal value over \((\omega,\beta)\) is equivalent to
minimizing jointly over the original and auxiliary variables. This
gives the lifted formulation
\[
\begin{aligned}
	\min_{\omega,\beta,u^\pm,s^\pm}\quad&
	\frac12\|\omega\|^2
	+
	C\sum_{i=1}^m
	\left[
	\frac{(u_i^+)^2}{2\delta}+s_i^+
	+
	\tau
	\left(
	\frac{(u_i^-)^2}{2\delta}+s_i^-
	\right)
	\right]\\
	\textnormal{subject to}\quad&
	u_i^++s_i^+
	\geq
	1-y_i(\omega^\top z_i+\beta),
	\qquad i=1,\ldots,m,\\
	&
	u_i^-+s_i^-
	\geq
	-1+y_i(\omega^\top z_i+\beta),
	\qquad i=1,\ldots,m,\\
	&
	u_i^\pm\geq0,\qquad
	s_i^\pm\geq0,
	\qquad i=1,\ldots,m.
\end{aligned}
\]
Let $\widetilde v:=(\omega,\beta,u^+,u^-,s^+,s^-)$
and let \(\Omega_{\rm hp}\) denote the feasible set above. Then
\[
\begin{aligned}
	f_1(\widetilde v)
	&:=
	\frac12\|\omega\|^2
	+
	C\sum_{i=1}^m
	\left[
	\frac{(u_i^+)^2}{2\delta}+s_i^+
	+
	\tau
	\left(
	\frac{(u_i^-)^2}{2\delta}+s_i^-
	\right)
	\right],\\
	f_2(\widetilde v)
	&:=
	\delta_{\Omega_{\rm hp}}(\widetilde v),
	\qquad
	A=I,
	\qquad
	b=0.
\end{aligned}
\]
Thus \(f_1\) is a finite-valued convex quadratic function and
\(f_2\) is proper, closed, and polyhedral convex.

For the Huber-loss SVR, using
\[
H_c(t)
=
\min_{u,s^+,s^-}
\left\{
\frac12u^2+c(s^++s^-):
u+s^+-s^-=t,\;
s^+,s^-\geq0
\right\},
\]
we obtain the equivalent lifted problem
\[
\begin{aligned}
	\min_{\omega,\beta,u,s^+,s^-}\quad&
	\frac12\|\omega\|^2
	+
	C\sum_{i=1}^m
	\left[
	\frac12u_i^2
	+
	c(s_i^++s_i^-)
	\right]\\
	\textnormal{subject to}\quad&
	u_i+s_i^+-s_i^-
	=
	y_i-\omega^\top z_i-\beta,
	\qquad i=1,\ldots,m,\\
	&
	s_i^+\geq0,\qquad
	s_i^-\geq0,
	\qquad i=1,\ldots,m.
\end{aligned}
\]
With $\widehat v:=(\omega,\beta,u,s^+,s^-)$
and \(\Omega_{\rm H}\) denoting the feasible set above, this problem
has the representation
\[
\begin{aligned}
	f_1(\widehat v)
	&:=
	\frac12\|\omega\|^2
	+
	C\sum_{i=1}^m
	\left[
	\frac12u_i^2+c(s_i^++s_i^-)
	\right],\\
	f_2(\widehat v)
	&:=
	\delta_{\Omega_{\rm H}}(\widehat v),
	\qquad
	A=I,
	\qquad
	b=0.
\end{aligned}
\]
Hence all three lifted models are instances of
\eqref{eq:quadratic-polyhedral-primal}.

\medskip
\noindent
\emph{Level-boundedness and error-bound conclusions.}
The pinball, \(\epsilon\)-insensitive pinball,
\(\epsilon\)-insensitive, and \(\epsilon\)-penalty scalar losses are
coercive on \(\mathbb R\). Therefore, bounded objective values imply
boundedness of \(\omega\) and of the corresponding residual at any
fixed training observation. Since the training set is nonempty,
\(\beta\) is also bounded. Hence these primal objectives are
level-bounded.

The least-squares SVC and SVR objectives are strongly convex, and the
bias-regularized hinge-loss SVC has a strongly convex quadratic term.
Hence these objectives are level-bounded.

For the lifted Huber-pinball SVC, bounded objective values imply
boundedness of \(\omega,u^\pm,s^\pm\). The two families of constraints
then imply that
\[
1-y_i(\omega^\top z_i+\beta)
\]
is bounded from both above and below. Since the training set is
nonempty and \(\omega\) is bounded, \(\beta\) is bounded. Hence the
lifted objective is level-bounded.

Similarly, for the lifted Huber-loss SVR, bounded objective values
imply boundedness of \(\omega,u,s^+,s^-\). The equality constraints
then imply boundedness of
\[
y_i-\omega^\top z_i-\beta,
\]
and hence boundedness of \(\beta\). Thus this lifted objective is
level-bounded.

The \(L_1\)-loss and \(L_2\)-loss SVC objectives need not be
level-bounded without an additional assumption on the training
labels. Nevertheless, both problems admit optimal solutions. If only
one class is represented, this is immediate by taking
\(\omega=0\) and choosing \(\beta\) so that all hinge losses vanish.
If both classes are represented, the corresponding objectives are
level-bounded and hence attain their minima. The exact lifting above
then shows that the lifted \(L_2\)-loss formulation also admits an
optimal solution.

Thus all the direct and lifted formulations satisfy the solvability
assumption of Proposition~
\ref{prop:quadratic-polyhedral-error-bounds}. Consequently, the primal
and dual local error bounds and the dual Luo--Tseng error bound hold
for all these formulations. For the level-bounded models identified
above, the primal Luo--Tseng error bound also follows. In particular,
it holds for the pinball-loss, \(\epsilon\)-insensitive
pinball-loss, and least-squares SVC models, for all the SVR models
considered above, for the lifted Huber-pinball and Huber-loss
formulations, and for the bias-regularized hinge-loss SVC. For the
\(L_1\)- and \(L_2\)-loss SVC models, the same conclusion holds
whenever the corresponding primal objective is level-bounded.
\hfill\(\square\)

\noindent \textbf{Proof of Proposition \ref{prop:strongly-convex-arbitrary-regularizer}
\noindent \textnormal{(i)}}\label{app:strongly-convex-primal-eb}
%
	Define
	\[
	T(v)
	:=
	\operatorname{prox}_P
	\bigl(v-\nabla f(v)\bigr),
	\qquad
	 r_P(v):=T(v)-v,
	\]
	and set  $w:=T(v). $ The optimality condition for the proximal mapping gives
	\[
	v-\nabla f(v)-w
	\in
	\partial P(w).
	\]
	Since $
	 r_P(v)=w-v,$
	the preceding inclusion can be rewritten as
	\[
	-\nabla f(v)- r_P(v)
	\in
	\partial P(w).
	\]
	Consequently,
	\begin{equation}
		\label{eq:strong-convex-subgradient-at-prox-point}
		\nabla f(w)-\nabla f(v)- r_P(v)
		\in
		\partial\Phi(w).
	\end{equation}
	
	Choose a neighborhood \(U\) of \(\bar v\) on which \(\nabla f\) is \(L\)-Lipschitz continuous. Since \(T\) is continuous and $
	T(\bar v)=\bar v,$
	there exists \(\rho>0\) such that  $v\in\mathbb B(\bar v,\rho) $ ensures that both \(v\) and \(T(v)\) belong to \(U\).
	
	For such a point \(v\), define
	\[
	g_w
	:=
	\nabla f(w)-\nabla f(v)- r_P(v).
	\]
	By \eqref{eq:strong-convex-subgradient-at-prox-point},  $g_w\in\partial\Phi(w).$ The local Lipschitz continuity of \(\nabla f\) yields
	\[
	\begin{aligned}
		\|g_w\|
		&\leq
		\|\nabla f(w)-\nabla f(v)\|
		+
		\| r_P(v)\|\\
		&\leq
		L\|w-v\|
		+
		\| r_P(v)\|\\
		&=
		(L+1)\| r_P(v)\|.
	\end{aligned}
	\]
	
	The function \(\Phi=f+P\) is \(\sigma\)-strongly convex. Therefore, its subdifferential is \(\sigma\)-strongly monotone. Since $	0\in\partial\Phi(\bar v) \ \text{and}\ 	g_w\in\partial\Phi(w),$
	strong monotonicity gives
	\[
	\langle g_w,w-\bar v\rangle
	\geq
	\sigma\|w-\bar v\|^2.
	\]
	By the Cauchy-Schwarz inequality,
	\[
	\sigma\|w-\bar v\|^2
	\leq
	\|g_w\|\|w-\bar v\|.
	\]
	It follows that
	\[
	\|w-\bar v\|
	\leq
	\frac{1}{\sigma}\|g_w\|
	\leq
	\frac{L+1}{\sigma}\| r_P(v)\|.
	\]
	
	Using \(w-v= r_P(v)\), we obtain
	\[
	\begin{aligned}
		\|v-\bar v\|
		&\leq
		\|v-w\|+\|w-\bar v\|\\
		&\leq
		\left(
		1+\frac{L+1}{\sigma}
		\right) 
		\| r_P(v)\|.
	\end{aligned}
	\]
	Since the primal solution set is the singleton	\(\{\bar v\}\), this proves
	\[
	\operatorname{dist}(v,\mathcal V)
	\leq
	\kappa_P\| r_P(v)\|,
	\qquad
	v\in\mathbb B(\bar v,\rho),
	\]
	with  $\kappa_P
	:=
	1+\frac{L+1}{\sigma}. $
	Thus statement \textnormal{(i)} holds.
	
		\noindent	\textnormal{(ii)}	Since the composite matrix in \eqref{eq:strongly-convex-general-regularizer} is \(A=I\),	Corollary~\ref{cor:primal-eb-kkt-msr-equivalence} shows that	\textnormal{(i)} and \textnormal{(ii)} are equivalent.
		
		\noindent	\textnormal{(iii)}	Since \(f\) is strongly convex,	Corollary~\ref{cor:dual-eb-kkt-msr-equivalence} shows that	\textnormal{(ii)} and \textnormal{(iii)} are equivalent.	Hence \textnormal{(i)}--\textnormal{(iii)} are equivalent and,	by \textnormal{(i)}, all three properties hold.
		
		\noindent	\textnormal{(iv)}	The primal objective \(\Phi=f+P\) is strongly convex and hence 	level-bounded. Since the primal   local error bound holds	at its unique solution, Proposition~\ref{prop:local-sublevel-relations}\textnormal{(iii)}	gives the primal Luo--Tseng error bound.
		
		\noindent	\textnormal{(v)}	Since \(A=I\) has full row rank, Corollary~\ref{cor:level-boundedness-sufficient-conditions}	implies that the dual objective \(d_P\) is level-bounded.	Together with \textnormal{(iii)},	Proposition~\ref{prop:local-sublevel-relations} \textnormal{(iv)}	gives the dual Luo--Tseng error bound.
		\hfill $\square$
		
\noindent
\textbf{Proof of Proposition
	\ref{prop:kkt-eb-rank-deficient-ls}}
\label{app:proof:kkt-eb-rank-deficient-ls}.

Since \(P\) is a norm, it is finite-valued, continuous, and
coercive. Hence the primal objective
\[
F(v)
=
\frac12\|Bv-y\|^2+P(v)
\]
is level-bounded and admits an optimal solution. Moreover, the
standard convex sum rule gives, for every
\(\bar v\in\mathcal V\),
\[
0
\in
B^\top(B\bar v-y)+\partial P(\bar v).
\]
Thus there exists
\[
\bar x
:=
-B^\top(B\bar v-y)
\in
\partial P(\bar v).
\]
By convex conjugacy, $\bar v\in\partial P^*(\bar x),$
so \((\bar v,\bar x)\in\mathcal S\). The Fenchel--Young
equalities then imply dual optimality and strong duality. Hence
Assumption~\ref{ass:basic-primal-dual-setting} is satisfied.

\medskip
\noindent
\textnormal{(i)}
In the identity composite representation used in
\eqref{eq:rank-deficient-regularized-ls}, the composite matrix is $ A=I.$
Hence \(A\) has full row rank. Since \(\nabla f_1(v)=B^\top(Bv-y)\) is globally Lipschitz continuous, Corollary~
\ref{cor:full-row-rank-kkt-msr-implies-primal-eb}
shows that metric subregularity of \(\mathcal M_P\) at
\((\bar v,\bar x)\) implies the primal local error bound at
\(\bar v\).

Since metric subregularity is assumed at every point of
\(\mathcal S\), the primal local error bound holds at every
\(\bar v\in\mathcal V\).

\medskip
\noindent
 \textnormal{(ii)}
We now establish the dual local error bound. Since \(B\) is not
assumed to have full column rank, \(f_1\) need not be strongly
convex, and therefore Corollary~
\ref{cor:strong-convexity-kkt-msr-implies-dual-eb}
cannot be applied directly.

Set $Q:=B^\top B, \ c:=B^\top y.$
Then
\[
f_1(v)
=
\frac12\langle v,Qv\rangle
-
\langle c,v\rangle
+
\frac12\|y\|^2.
\]
Since $\operatorname{range}(Q)
=
\operatorname{range}(B^\top),$
we have \(c\in\operatorname{range}(Q)\). The Fenchel conjugate of
\(f_1\) is
\begin{equation}
	\label{eq:rank-deficient-ls-conjugate}
	f_1^*(s)
	=
	\frac12
	\left\langle
	s+c,Q^\dagger(s+c)
	\right\rangle
	-
	\frac12\|y\|^2
	+
	\delta_{\operatorname{range}(Q)}(s).
\end{equation}
Hence $
\operatorname{dom}f_1^*
=
\operatorname{range}(Q),$
and, for every \(s\in\operatorname{range}(Q)\),
\begin{equation}
	\label{eq:rank-deficient-ls-conjugate-subdiff}
	\partial f_1^*(s)
	=
	Q^\dagger(s+c)+\ker Q.
\end{equation}

Since \(P\) is a norm, its conjugate is the indicator function of
the closed unit ball \(C_P^*\) of the dual norm:
\[
P^*=\delta_{C_P^*}.
\]
The set \(C_P^*\) is compact and convex, and
\[
0\in\operatorname{int}C_P^*.
\]
On the other hand,
\[
\operatorname{dom}(f_1^*(-\cdot))
=
\operatorname{range}(Q),
\]
which is a linear subspace. Therefore,
\[
0
\in
\operatorname{ri}(\operatorname{dom}P^*)
\cap
\operatorname{ri}
\bigl(
\operatorname{dom}(f_1^*(-\cdot))
\bigr).
\]
The convex sum rule, together with the exact chain rule for
\(x\mapsto-x\), gives
\begin{equation}
	\label{eq:rank-deficient-dual-subdiff}
	\partial d(x)
	=
	\partial P^*(x)
	-
	\partial f_1^*(-x),
	\qquad
	x\in\mathbb R^n.
\end{equation}

Fix any \(\bar x\in\Lambda\). Suppose, to the contrary, that the
dual local error bound fails at \(\bar x\). Then, for every integer
\(k\geq1\), there exists
\[
x^k\in\mathbb B(\bar x,1/k)
\]
such that
\begin{equation}
	\label{eq:dual-eb-failure-rank-deficient}
	\operatorname{dist}(x^k,\Lambda)
	>
	k\,r_D(x^k).
\end{equation}
Since
\[
\operatorname{dist}(x^k,\Lambda)
\leq
\|x^k-\bar x\|
\leq
\frac1k,
\]
we obtain
\[
r_D(x^k)
<
\frac1{k^2}
\to0.
\]
In particular, \(r_D(x^k)<+\infty\).

For each \(k\), choose $
e^k\in\partial d(x^k)$
such that
\[
\|e^k\|=r_D(x^k).
\]
By \eqref{eq:rank-deficient-dual-subdiff}, there exist $u^k\in\partial P^*(x^k),\  v^k\in\partial f_1^*(-x^k)$
such that
\begin{equation}
	\label{eq:dual-decomposition-rank-deficient}
	e^k=u^k-v^k.
\end{equation}

We first show that \(\{v^k\}\) is bounded. Suppose otherwise.
Passing to a subsequence if necessary, we may assume that
\[
\|v^k\|\to+\infty,
\qquad
\frac{v^k}{\|v^k\|}\to z,
\qquad
\|z\|=1.
\]
By
\eqref{eq:rank-deficient-ls-conjugate-subdiff},
\[
v^k
=
Q^\dagger(c-x^k)+z^k,
\qquad
z^k\in\ker Q.
\]
Since \(x^k\to\bar x\), the sequence $\{Q^\dagger(c-x^k)\}$ is bounded. Therefore,
\[
z\in\ker Q.
\]

Moreover,
\[
\partial P^*(x^k)
=
N_{C_P^*}(x^k),
\]
and \eqref{eq:dual-decomposition-rank-deficient} gives $u^k=v^k+e^k.$
Since \(N_{C_P^*}(x^k)\) is a cone,
\[
\frac{u^k}{\|v^k\|}
\in
N_{C_P^*}(x^k).
\]
Because \(e^k\to0\),
\[
\frac{u^k}{\|v^k\|}
\to z.
\]
The graph of the normal-cone mapping to a closed convex set is
closed. Hence
\[
z\in
N_{C_P^*}(\bar x)
=
\partial P^*(\bar x).
\]

Choose any \(\widehat v\in\mathcal V\). Since
\(\bar x\in\Lambda\) and
\(\mathcal S=\mathcal V\times\Lambda\), we have
\[
Q\widehat v-c+\bar x=0,
\qquad
\widehat v\in N_{C_P^*}(\bar x).
\]
Since \(N_{C_P^*}(\bar x)\) is a convex cone and $
z\in N_{C_P^*}(\bar x)\cap\ker Q,$
it follows that, for every \(t\geq0\), $
\widehat v+t z
\in
N_{C_P^*}(\bar x)$
and
\[
Q(\widehat v+t z)-c+\bar x=0.
\]
Thus
\[
(\widehat v+t z,\bar x)\in\mathcal S
\]
for every \(t\geq0\). Hence
\[
\widehat v+t z\in\mathcal V,
\qquad
t\geq0.
\]
This contradicts the boundedness of \(\mathcal V\), which follows
from the level-boundedness of \(F\). Therefore \(\{v^k\}\) is
bounded.

Passing to a further subsequence if necessary, suppose that $
v^k\to v^*.$
Since \(e^k\to0\),
\eqref{eq:dual-decomposition-rank-deficient} also gives
\[
u^k\to v^*.
\]
The graphs of \(\partial f_1^*\) and \(\partial P^*\) are closed,
and therefore
\[
v^*\in\partial f_1^*(-\bar x),
\qquad
v^*\in\partial P^*(\bar x).
\]
By convex conjugacy and differentiability of \(f_1\),
\[
B^\top(Bv^*-y)+\bar x=0.
\]
Hence
\[
(v^*,\bar x)\in\mathcal S.
\]

By assumption, \(\mathcal M_P\) is metrically subregular at
\((v^*,\bar x)\) for \(0\). Thus there exist
\(\kappa>0\) and \(\rho>0\) such that
\[
\operatorname{dist}((v,x),\mathcal S)
\leq
\kappa\,
\operatorname{dist}(0,\mathcal M_P(v,x))
\]
whenever $(v,x)\in
\mathbb B((v^*,\bar x),\rho).$
For all sufficiently large \(k\),
\[
(v^k,x^k)
\in
\mathbb B((v^*,\bar x),\rho).
\]

Since $
v^k\in\partial f_1^*(-x^k),$
convex conjugacy gives
\[
B^\top(Bv^k-y)+x^k=0.
\]
Furthermore,
\[
\begin{aligned}
	\operatorname{dist}
	\bigl(v^k,\partial P^*(x^k)\bigr)
	&\leq
	\|v^k-u^k\|=
	\|e^k\|=
	r_D(x^k).
\end{aligned}
\]
Consequently,
\[
r_{\rm KKT}(v^k,x^k)
\leq
r_D(x^k).
\]
Metric subregularity therefore yields
\[
\begin{aligned}
	\operatorname{dist}(x^k,\Lambda)
	&\leq
	\operatorname{dist}((v^k,x^k),\mathcal S)\leq
	\kappa r_{\rm KKT}(v^k,x^k)\leq
	\kappa r_D(x^k),
\end{aligned}
\]
which contradicts
\eqref{eq:dual-eb-failure-rank-deficient} for every sufficiently
large \(k\). Hence the dual local error bound holds at
\(\bar x\). Since \(\bar x\in\Lambda\) was arbitrary, it holds at
every dual solution.

\medskip
\noindent
 \textnormal{(iii)} As observed above, \(F\) is level-bounded because \(P\) is a norm.
Moreover,
\[
\operatorname{dom}d
\subseteq
\operatorname{dom}P^*
=
C_P^*.
\]
Since \(C_P^*\) is compact and \(d\) is closed, every finite lower
level set of \(d\) is compact. Hence \(d\) is level-bounded.

By parts \textnormal{(i)} and \textnormal{(ii)}, the primal and
dual local error bounds hold at every primal and dual solution,
respectively. Proposition~
\ref{prop:local-sublevel-relations}
\textnormal{(iii)}--\textnormal{(iv)} therefore yields the primal
and dual Luo--Tseng error bounds.
\hfill $\square$
\end{appendices}

\bibliography{sn-bibliography_928a}

\end{document}